\documentclass[a4paper,11pt]{article}

\usepackage{graphicx}
\usepackage{amsmath,amsthm,amssymb,amsfonts}
\usepackage{mathtools,booktabs}
\usepackage{authblk}
\usepackage{tabularx}
\usepackage{bm}

\newcommand{\vb}{\mathbf{b}}
\newcommand{\vc}{\mathbf{c}}
\newcommand{\vd}{\mathbf{d}}
\newcommand{\ve}{\mathbf{e}}
\newcommand{\vl}{\mathbf{l}}
\newcommand{\vu}{\mathbf{u}}
\newcommand{\vv}{\mathbf{v}}
\newcommand{\vs}{\mathbf{s}}
\newcommand{\vt}{\mathbf{t}}
\newcommand{\vw}{\mathbf{w}}
\newcommand{\vx}{\mathbf{x}}

\newcommand{\vz}{\mathbf{z}}
\newcommand{\allzeros}{\mathbf{0}}
\newcommand{\allones}{\mathbf{1}}
\newcommand{\vA}{\mathbf{A}}
\newcommand{\vC}{\mathbf{C}}
\newcommand{\vD}{\mathbf{D}}
\newcommand{\vE}{\mathbf{E}}
\newcommand{\vF}{\mathbf{F}}
\newcommand{\vI}{\mathbf{I}}
\newcommand{\vL}{\mathbf{L}}
\newcommand{\vP}{\mathbf{P}}
\newcommand{\vQ}{\mathbf{Q}}
\newcommand{\vT}{\mathbf{T}}
\newcommand{\vU}{\mathbf{U}}
\newcommand{\vV}{\mathbf{V}}
\newcommand{\vW}{\mathbf{W}}
\newcommand{\vlambda}{\bm{\lambda}}
\newcommand{\veta}{\bm{\eta}}
\newcommand{\vsigma}{\bm{\sigma}}
\newcommand{\vSigma}{\bm{\Sigma}}

\newcommand{\valpha}{\bm{\alpha}}
\newcommand{\vbeta}{\bm{\beta}}

\newcommand{\vmu}{\bm{\mu}}
\usepackage{color}
\usepackage{mathptmx}      
\usepackage{pgfplots}
\usepackage{tikz}
\usepackage{tikz-3dplot}
\usepackage{latexsym}
\usepackage[ruled,vlined]{algorithm2e}
\usepackage[flushleft]{threeparttable}

\newtheorem{theorem}{Theorem}
\newtheorem{assumption}{\bf Assumption}
\newtheorem{corollary}{Corollary}
\newtheorem{definition}{Definition}
\newtheorem{lemma}{Lemma}
\newcommand{\norm}[2] {\|{#1}\|_{#2}}
\newcommand{\abs}[1] {\left|{#1}\right|}

\DeclareMathOperator*{\conv}{conv}
\DeclareMathOperator*{\ext}{ext}

\DeclareMathOperator*{\argmin}{arg\,min}

\def\rit{\mathbb{R}}

\begin{document}
\title{Multipolar Robust Optimization}

\author[1]{Walid Ben-Ameur \thanks{walid.benameur@telecom-sudparis.eu}}
\author[2]{Guanglei Wang}
\author[2]{Adam  Ouorou}
\author[3]{Mateusz \.Zotkiewicz}

\affil[1]{Samovar, T\'el\'ecom SudParis, CNRS, Universit\'e Paris-Saclay}
\affil[2]{Orange Labs Research, France}
\affil[3] {Institute of Telecommunications, Warsaw University of Technology, Poland}
\maketitle

\begin{abstract}
We consider linear programs involving uncertain parameters and propose a new
tractable robust counterpart which contains and generalizes several other
models including the existing \emph{Affinely Adjustable Robust Counterpart} and
the \emph{Fully Adjustable Robust Counterpart}. It consists in selecting a set
of \emph{poles} whose convex hull contains some projection of the uncertainty
set, and computing a recourse strategy for each data scenario as a convex
combination of some optimized recourses (one for each pole). We show that the
proposed \emph{multipolar robust counterpart} is tractable and its complexity is controllable.
Further, we show that under some mild assumptions, two sequences of upper and lower bounds
 converge to the optimal value of the fully adjustable robust
counterpart. To illustrate the approach, a robust problem related to lobbying
under some uncertain opinions of authorities is studied. Several numerical
experiments are carried out showing the advantages of the proposed robustness
framework and evaluating \emph{the benefit of adaptability}. 

\end{abstract}

\section{Introduction}
\label{intro}

Uncertainty in optimization parameters arises in many applications due to the
difficulty to measure data or because of their variability.  
To deal with uncertainty, there are mainly two approaches:  stochastic
optimization and robust optimization. In the first case, some probabilistic
assumptions are made about the uncertain data \cite{BiLo,Dan56,KaWa}.  One is
then interested in computing a solution optimizing some moments of  random
variables depending on the data. Another variant, known as chance constrained
programming \cite{ChCo}, consists in imposing that some constraints are
satisfied only with some probability. 

Robust optimization is a more recent approach dealing with uncertainty. It
does not require specifications of the exact distribution of the problem's
parameters. Roughly speaking, uncertain data are assumed to belong to a known
compact set, called uncertainty set, and we aim at finding a solution that is \emph{immunized}
against all possible realizations in the uncertainty set.  An early
contribution related to robust optimization is the work of Soyster \cite{SOY}
followed by intensive investigations in the last 20 years starting with
\cite{Bental98,ELG} in the context of convex optimization and the book
\cite{KOUV} dealing with discrete optimization. Almost at the same time, and
in an independent way, a lot of work was initiated in \cite{Fin97} and
\cite{DUF98} on robust optimization in communication networks   dealing with uncertain traffic
matrix, see \cite{BOZ12} for a survey. 

Robust optimization and stochastic programming are related in numerous ways.
For example, using some knowledge about the distribution of uncertain data, it
is sometimes possible to define an uncertainty set in such a way that the
robust solution is an approximated solution of a chance constrained problem
(see, e.g., \cite{BenEl-Nem07a,BER2011b} for details and references). An
approach combining robust optimization and stochastic programming  consists in
computing solutions that are distributionally robust where the distribution of
parameters is assumed to vary within some set (for example, when the mean and
the covariance matrix are known) (see, e.g., \cite{ELG03,GOH}).

The definition of the uncertainty set is a critical issue since a bad choice
might lead to very expensive solutions. One way to alleviate overconservatism
of the robust approach is to assume that a subset of the decision variables are
adjustable on the realization of the uncertain data.  Let us for example
consider the following linear problem
\begin{align*}
    \min \hspace{1em}& \vc^{T}\vx\\
    \textrm{s.t.}\hspace{1em}& \vA \vx\le \vb,\\
& \vx\in\rit^{n},
\end{align*}
involving uncertain parameters. We assume that $\vx$ is partitioned as
$\vx=(\vu,\vv)$, where $\vu$ represents the non-adjustable and
$\vv$ the adjustable variables. The robust counterpart of this uncertain
problem under consideration reads 
\begin{equation}\label{eq:pb} \tag{FARC}
\begin{split}
    \min\limits_{u,v} \hspace{1em}& \vc^T\vu\\
    \textrm{s.t.}\hspace{1em}& \vU\vu+\vV\vv(\xi)\le \vb, ~
    \lbrack\vU,\vb\rbrack \in \Xi,
\end{split}
\end{equation}
where the uncertain parameters are $\vU\in \rit^{m\times n}$ and $\vb\in
\rit^m$, while $\vV$ and $\vc$ are assumed to be known. We denote by 
$\xi\equiv \lbrack \vU,\vb\rbrack \in \Xi$ the uncertain parameters  belonging to the  uncertainty set  $\Xi$ assumed to be 
compact, convex and with a non-empty interior.  $\xi$ will be considered as a vector in the rest of the paper.

The non-adjustable variables are sometimes interpreted ``here and now"
variables, while the adjustable ones can be seen as ``wait and see" variables.
This robust counterpart above is generally called \emph{fully-adjustable robust
counterpart} (FARC). FARC is sometimes called the dynamic robust
counterpart since $\vv$ depends on $\xi$. FARC can be seen as a  two-stage
optimization problem where $\vu$ are the first-stage variables and $\vv$ are the
second-stage variables.     

If variables $\vv$ are also static, then FARC simply becomes the standard
\emph{static robust counterpart} denoted by SRC. Cases where FARC and SRC are
equivalent have  been pointed out in \cite{BenGorGus04} where it is shown that
adaptability does not lead to any improvement in the constraint-wise
uncertainty case. Still, FARC is generally much less conservative than SRC.
In other words, there is generally some \emph{benefit of adaptability}. Solving
FARC is, unfortunately difficult in general cases as shown by many authors
\cite{BenGorGus04,CHE,MIN10}. Another concern related to FARC is the inherent
difficulty of implementing the solution $\vv(\xi)$ in a practical way. 

 To get a tractable optimization problem and also to alleviate
some of the overconservatism of SRC, an affinely adjustable approach was
proposed in \cite{BenGorGus04}, where the adjustable variables $\vv$ are not
fully adaptable (dynamic), but are assumed to depend on the uncertain data
$\xi$ in an affine way: 
\begin{equation*} 
\vv(\xi)= \vw+\vW\xi,\quad \xi\equiv\lbrack \vU, \vb\rbrack\in\Xi,
\end{equation*}
where $\vw$ and the elements of matrix $\vW$ are new decision variables
(a.k.a. affine decision rules). The induced formulation is called
\emph{affine-adjustable robust counterpart} (AARC).  An affine approach was
also independently proposed for network optimization problems where the traffic
matrix is supposed to be uncertain and the way how traffic is splitted through
network's paths is optimized \cite{WAL03,WAL05,OuVi07}. Further developments
appeared in \cite{BKOV09,Ouo12}.

Applying affine decision rules naturally leads to less expensive solutions than
those obtained by the static approach. The performance gap quantified by the
difference between optimum of AARC and the optimum of FARC was discussed for
robust linear  problems with right-hand-side uncertainty in
\cite{BER2012,BER2010b}. One of the results of \cite{BER2012} states that AARC
is equivalent to FARC when the uncertainty set is a
simplex. Some tight approximation bounds relating the optimum of AARC to that
of FARC in the right-hand-side uncertainty case are also given there. 

Related investigation on problems with some special uncertainty sets (integer
sublattices of the unit hypercube) are discussed in
\cite{iancu2013supermodularity}, where they provide sufficient conditions such
that the associated affinely adjustable decision rules lead to exact optimum of
FARC.

The affine approach is related to the well-known linear or
first-order decision rules used in the context of multi-stage stochastic
optimization \cite{GAR}.
Linear decisions rules were also used in \cite{KUHN} in the context of
stochastic programming not only to get upper bounds (as done above) but also to
get lower bounds by properly approximating the dual problem using linear
decision rule.

As observed in \cite{CHEN}, even though AARC has been successfully applied to
several problems, its performance might be unsatisfactory under situations where
the adjustable variables exhibit high nonlinearity in terms of the uncertain
parameters. This led to some extensions of the affine approach in
\cite{CHEN08,CHEN} (see also references therein) by reparametrizing the
uncertainties and then applying the affinely adjustable approach. Roughly
speaking, a new set of variables is introduced (for example the positive and
the negative parts of the original uncertainties), and the adjustable variables
are assumed to affinely depend on the new set of parameters. A similar idea is
also proposed in \cite{BER2010b} in the context of one-dimensional constrained
multistage robust optimization.

Other extensions of affine decision rules have been proposed in literature. In
\cite{BER2011a} polynomial recourse actions are considered where $\vv$ is
expressed as a polynomial in uncertainty parameters with degree no larger than
a fixed constant. The complexity of the robust counterpart problem is then
related to testing the positivity of a polynomial. Using some recent results in
algebraic geometry stating that under mild conditions, a positive polynomial can be
expressed as a sum of squares (not a priori bounded), the robust counterpart is
approximated by considering sums of squares of degree no larger than a fixed
constant.  As a sum of squares can be represented by a semidefinite programming
\cite{LASS}, the proposed robust counterpart can be efficiently handled
\cite{BER2011a}.  

Another robust approach dealing with uncertainty, termed as \emph{multi-static}
approach in \cite{BOZ12}, was proposed and studied in
\cite{WAL07,BenZet11,BenZet10}. It consists in partitioning the uncertainty set
$\Xi$ into a finite number of subsets $\Xi_1,\dots,\Xi_p$ and using a recourse
action $\vv_i$ for each subset $\Xi_i$. In other words, if $\xi \in \Xi_i$,
then we take $\vv = \vv_i$. The recourse actions $\vv_i$ are of
course subject to optimization.     
A quite close idea is proposed in
\cite{BER2010}, where it was called \emph{finite adaptability}. The performance
of finite adaptability in a fairly general class of multi-stage
stochastic and adaptive optimization problems was investigated in
\cite{bertsimas2011geometric}. 

One can also combine finite adaptability and the affinely-adjustable approach
by partitioning the uncertainty set into some subsets and considering some
optimized specific affine decision rules for each subset. This was also
considered in   \cite{WAL07,BOZ12} in the context of network design problems.
This type of adaptability  might also be called \emph{piecewise-affine
adaptability}.  Piecewise-affine rules  were also considered in several
other papers such as \cite{BEM,GOH}.

While a great number of proposals in robust optimization have appeared, there are
still challenges. First, to the best of our knowledges, none are general
enough to encompass static robustness, affinely adjustable robustness and fully
adjustable robustness. Second, as observed in \cite{BER2011a}, there is no
systematic way to influence the trade-off between the performance of the
resulting policies and the computational complexity required to obtain them.
Third, the uncertain parameters of an optimization problem can be sometimes
difficult to observe. In several applications, only a subset of such parameters
or some aggregates of them can be observed.   

The objective of this paper is to provide a framework addressing those
challenges at the same time. Our contributions are four-fold:
\begin{enumerate}
\item A novel approach. We propose a hierarchical and convergent framework of
    adjustable robust optimization -- \emph{multipolar robust approach}, which
    generalizes notions of static robustness, affinely adjustable
    robustness, fully adjustable robustness and fill the gaps in-between. As a
    byproduct, a new way to look at the affine adaptability is proposed. The
    result of \cite{BER2012} stating that  affine rules are optimal  when the
    uncertainty set is a simplex is also obtained as a consequence of the
    multipolar approach. 
\item A comprehensive analysis. We show that the \emph{multipolar robust
counterpart} is tractable by either a cut generation procedure or a compact
formulation. Further, we prove that the multipolar approach can generate a
sequence of upper bounds and a sequence of lower bounds at the same time and both
sequences converge to the robust value of FARC under some mild assumptions.
\item A general constructive algorithm of \emph{pole-sets}. The multipolar
approach is based on some tools related related to the uncertainty set, that we
term as pole-sets. For their construction, we start with a simplex and then
compute the best homothetic transformation of this simplex to allow it to
enclose a given convex set. An efficient algorithm is proposed to compute such
homothetic set. As a byproduct, we provide a very simple proof of the geometric
results of \cite{Nevskii2011} related to hypercubes. The pole-sets obtained
after this homothetic transformation are then improved using a tightening
procedure. 
\item An application. To numerically illustrate the multipolar approach, a
lobbying problem is considered where a lobby aims to minimize the budget needed
to convince a set of voters taking into account a reasonable opinion dynamics
model under some uncertainty. The benefit of adaptability is clearly shown for
this problem.  

\end{enumerate}

{\bf Outline.} In Section \ref{sec:concept}, we present the concept and
ingredients of multipolar robust optimization and show that static robustness,
affinely adjustable robustness, fully adjustable robustness are special cases
of multipolar robust framework. In Section \ref{sec:analysis}, we discuss the
tractability, the monotonicity and the convergence of the proposed approach.  A
simple illustrative example is described in Section \ref{sec:example}.  In
Section \ref{sec:poles}, we propose algorithms for pole-set generation. Section
\ref{sec:problem} is dedicated to a numerical example on a lobbying problem
under several uncertainty scenarios. Finally, concluding remarks follow in
Section \ref{sec:conclusion}.

{\bf Notation.}
Throughout this paper, we use $\Xi$ to represent a compact convex uncertainty
set  and $\xi$ to denote a member of $\Xi$. We use $\vI$ to denote the identity
matrix. Vectors and matrices are marked in bold, and their scalar components
are presented in italic.  Given any matrix $\vC$,  ${\vC}^T$ denotes its
transpose. We also use $[\vC,\vD]$ to denote the matrix where $\vC$ and $\vD$
are concatenated by columns assuming they have the same number of rows.
Similarly, $(\vC,\vD)$  denotes the matrix obtained by  row concatenation of
two matrices $\vC$ and $\vD$ when they have the same number of columns. Observe
that $\vv = (v_1,\ldots, v_n)$ is then a vector and $[v_1,\ldots,v_n] = \vv^T$.  We
use $\delta_{ij}$ to represent the Kronecker's delta function, where
$\delta_{ij}=1$ if $i=j$, $0$ otherwise.
For a set $S\in \rit^n$, we use $\ext(S)$ to represent the set of its extreme
points,  $\conv S$ to represent its convex hull and $\dim(S)$ to denote its
dimension. If $S$ is finite, we use $\abs{S}$ to represent its cardinality. We
also use the standard notation for usual norms: $\norm{\cdot}{\infty}$ for the
infinity norm,  $\norm{\cdot}{1}$ for the Manhattan norm and $\norm{\cdot}{2}$ for the Euclidean norm.

\section{The multipolar robust optimization concept} \label{sec:concept}

In this section, we introduce the main ingredients of multipolar robustness and then
setup the multipolar robust counterpart as a novel approximation of
\ref{eq:pb}.

{\bf Shadow matrix.} Like other robust approaches, multipolar approach is also
based on an uncertainty set $\Xi$. In addition, we consider a matrix associated
with certain operations on the uncertain information, which can be data
aggregation, filtering, and selection. Note that these operations can either be
natural or artificial. Natural operations are induced by the difficulty of
measurements or shortage of data. For example, in communication networks,
traffic flows are usually observed in an aggregated manner (the consequence of
aggregating uncertain demands from multiple origin-destination pairs).
Nevertheless, adjustable recourse actions should be implemented based the
observed partial information.  On the other hand, artificial operations can be
certain techniques to control the complexity of the multipolar robust
counterpart, as explained in the concluding remarks of this section. We call
the associated matrix of an operation \emph{shadow matrix} since the operation
either reduces the size of the multipolar robust counterpart or is a direct
consequence of observations. We use $\vP\in \rit^{n_0\times \dim(\Xi)}$ to
denote a shadow matrix, where $n_0$ is the dimension of the shadow (i.e., the
resulting partial information) and $\dim(\Xi)$ is the dimension of the
uncertainty set $\Xi$. The resulting partial information is defined by 
\begin{align}\label{eq:shadow}
\Xi_{P}:=\vP\Xi\equiv\{\vP\xi, \xi\in \Xi\}.
\end{align}
 When $\vP$ is identity matrix,
we have a complete measure of uncertainty.

We will  assume that $\vP$  is full row rank matrix.
Consequently, $\Xi_{P}$ is also compact, convex and has a non-empty interior. 

{\bf Pole-set.} A key component of the multipolar approach is a finite set
of \emph{poles}, which are given vectors in the range space of the
shadow matrix. We denote by $\Omega$ such a \emph{pole-set}. We say that
$\Omega$ is a pole-set of  $\Xi_P$ iff for any $\xi\in\Xi$, $\vP\xi$ belongs to
the convex hull of $\Omega$ (a convex combination of poles) denoted by $\conv
\Omega$. Given a set $\Xi_P$, a collection of pole-sets of $\Xi_P$
is defined as
\begin{align}\label{eq:feasibleOmega}
\mathcal{F}_{\Xi_P} :=\left\{\Omega: \Xi_P \subseteq \conv \Omega \right\}.
\end{align}
Obviously, extreme points of $\Xi_P$ form a pole-set, i.e., $\ext(\Xi_P) \in
\mathcal{F}_{\Xi_P}$. 

{\bf Multipolar robust counterpart.}
We now setup the multipolar robust counterpart w.r.t. an uncertainty set $\Xi$,
a shadow matrix $\vP$, a pole-set $\Omega \in \mathcal{F}_{\Xi_P}$. 
For each $\xi\in \Xi$, we consider a weight $\lambda^{\xi}_{\omega}$ for each pole
$\omega$ in $\Omega$. Then, for each scenario $\xi\in\Xi$, the following system
has a solution 
\begin{equation} \label{eq:policy}
\begin{array}{l}
\sum\limits_{\omega\in\Omega}\lambda^{\xi}_{\omega}\omega=\vP\xi,\\ \\[-2mm]
\sum\limits_{\omega\in\Omega}\lambda^{\xi}_{\omega}=1,\\
\\[-2mm]
\lambda^{\xi}_{\omega}\ge 0,\;\;\omega\in\Omega.
\end{array}
\end{equation}
Let $\Lambda_{\xi}$ be the set of weight vectors $\vlambda^{\xi}$ satisfying the
above system for a given $\xi\in\Xi$. In the considered paradigm, each pole is
associated with a recourse action, and the recourse action in the presence of
$\xi\in \Xi$ is approximated by a convex combination of the recourse actions
associated with the poles. Specifically, let vector $\vv_{\omega}$ be the
recourse action associated with pole $\omega$ in the above system. We require
the adjustable variables $\vv(\xi)$ to be restricted to
\begin{equation}\label{mpconstr}
    \vv(\xi) = \sum\limits_{\omega\in\Omega}\lambda^{\xi}_{\omega}\vv_{\omega},
\end{equation}
where $\vlambda^{\xi}\in \Lambda_\xi$. We can readily present the \emph{multipolar
robust counterpart} defined by
\begin{align}
    \Pi_\Xi(\vP,\Omega) = &\min \limits_{\vu,\vv} \hspace{1em}\vc^T\vu \tag{MRC}\label{mpro}\\
    ~&\textrm{s.t.}\hspace{1em}\vU\vu+\vV\sum\limits_{\omega\in\Omega}\lambda^{\xi}_{\omega}\vv_{\omega}\le
    \vb,~\xi\in\Xi,\;\vlambda^{\xi}\in\Lambda_{\xi}.  \label{eq:counterpart}
\end{align}
Following the spirits of robust optimization, the multipolar robust counterpart
\eqref{mpro} seeks a pair of
non-adjustable solution $\vu$ and a set of recourse actions related to poles 
$\vv_\omega, \omega \in \Omega$ such that the objective function is minimized
while hedging against the uncertainty set $\Xi$. In brief, given $\Xi$, the
multipolar robust approach can bee seen as a set function
of a pole-set $\Omega$ and a shadow matrix $\vP$. We denote the function by
\begin{align*}
\Pi_\Xi: \rit^{n_0\times \dim(\Xi)} \times \mathcal{F}_{\Xi_P} \ni 
\left(\vP,\Omega\right) \mapsto \Pi_{\Xi}\left(\vP,\Omega\right)\in \rit
\end{align*}
and call $\Pi_\Xi\left(\vP,\Omega\right)$ \emph{multipolar robust value} w.r.t. $\left(\vP,\Omega\right)$. Also,
we call $\left(\vu,\vv\right)$ \emph{multipolar solution}.

To conclude this section, we add few remarks on the concept of shadow matrix
and pole-set to clarify the motivation behind  these ingredients.
\begin{itemize}
    \item Note that by \eqref{eq:counterpart}, the solution is protected
        against the considered uncertainty $\Xi$. Neither a shadow matrix $\vP$ nor a pole-set
        $\Omega$ changes the uncertainty set, so $\vP$ and $\Omega$ are not used
        to approximate the uncertainty set. 
    \item Observe that $1+ \dim(\Xi_P) \le \abs{\Omega}$, so we can use 
        the shadow matrix $\vP$ to reduce the number of recourse actions and
        therefore the number of variables of \eqref{eq:counterpart}. Reducing
        the number of poles leads to an MRC which is easier to
        solve as will be shown in Section \ref{sec:analysis}.

    \item In several applications, after data is revealed, the adjustable
        variables should be quickly chosen and used. This is fortunately easy
        to do in the multipolar robust framework since the only thing to do is
        to find the coefficients $ \lambda^{\xi}_{\omega}$ and use them to
        combine the already computed recourse vectors $\vv_\omega, \omega \in
        \Omega$.
\end{itemize}
\subsection{Special cases} \label{sec:special}

We show in this section that \ref{mpro} generalizes SRC, AARC, and FARC by
different settings of pole-sets and recourse actions associated with poles. 

First, we show that SRC is a special case of \ref{mpro}. Imposing $\vv_{\omega} =
\vv_{\omega'}$ for any pair of $\omega$ and $\omega'$ belonging to $\Omega$ leads
to  $\vv(\xi)=\vv_{\omega}, \forall \xi\in \Xi$, which means that the recourse
action is static. Another way to get SRC is to impose that $\vP$ is a null
matrix having one row (relaxing in this case the full row rank constraint related to $P$) and  $\Omega$ contains just the zero vector.

Second, we show that FARC is a special case of \ref{mpro}. Let $\Omega$ be the
set of extreme points of $\Xi$ and $\vP=\vI$. Then $\conv\Omega=\Xi$, that is for
$\xi\in\Xi$, there exists $\vlambda^{\xi}\ge \allzeros$  such that
$\sum\limits_{\omega\in\Omega}\lambda^{\xi}_{\omega}=1$ and
$\sum\limits_{\omega\in\Omega}\lambda^{\xi}_{\omega}\omega=\xi$. By linearity
of inequalities \eqref{eq:counterpart}, imposing that
$\vU\vu+\vV\vv_{\omega}\le \vb$ for
each extreme point $\omega \in \Omega$ is necessary and sufficient to ensure
the satisfaction of all inequalities \eqref{eq:counterpart} for each $\xi \in
\Xi$. We get here the fully adjustable case representing the best that we can
obtain for this problem since  it is equivalent to assuming that $\vv$ can vary
with no restrictions.  Note that if the number of extreme points of  $\Xi$ is
limited, then the robust optimization counterpart can be efficiently solved.
However, if the number of extreme points of $\Xi$ is non-polynomial, the
problem is generally difficult (as already mentioned in Section \ref{intro},
see for example \cite{BenGorGus04,CHE,MIN10}).

Third, we show that AARC can also be generalized by \ref{mpro} by proving the
following theorem. 

\begin{theorem}\label{thm1:general}
Let $\Omega \in \mathcal{F}_{\Xi_P}$  such that $|\Omega| = 1 + \dim({\vP \Xi})$. Then the optimal solution of the corresponding MRC problem is exactly the best solution that is affine in $\vP \xi$. 
\end{theorem}
\begin{proof}
Since $\vP \Xi$ has non-empty interior,
$\vP \Xi \subset \conv \Omega$ and 
$|\Omega| = 1 + \dim({\vP \Xi})$,   the elements of $\Omega$ are affinely independent. 
Let $d=\dim({\vP \Xi})$ and assume  $\Omega = \{ \omega^{(1)},\dots,
\omega^{(d+1)}\}$. The shadow matrix
$\vP$ is here the identity matrix. 
Consider matrix $D$ obtained by taking vectors $\omega^{(i)}$ as
columns and adding a final line containing only coefficients equal to $1$.
\begin{align*}
    \vD=\begin{pmatrix} \omega^{(1)}_1 &\hdots & \omega^{(d+1)}_1 \\ \vdots &\vdots &\vdots
\\ \omega^{(1)}_d &\hdots &\omega^{(n+1)}_d \\ 1 &\hdots& 1 \end{pmatrix}. 
\end{align*}
Observe that $\vD$ is a non-singular square matrix of size $(d+1)$. 

Given any $\xi$, there are unique coefficients $\lambda^{\xi}_{\omega}$ such
that $\vP \xi = \sum\limits_{\omega\in\Omega}\lambda^{\xi}_{\omega}\omega$  and $
\sum\limits_{\omega\in\Omega} \lambda^{\xi}_{\omega}=1$. This can be written as
$\left(\begin{array}{l} \vP \xi , 1  
\end{array} \right) = \vD  \vlambda^{\xi}$ where  $\vlambda^{\xi}$ is the vector
whose components are the  $\lambda^{\xi}_{\omega}$ for $\omega \in \Omega$.  
This immediately implies that 
$\vlambda^{\xi} = \vD^{-1}  \left(\begin{array}{l} \vP \xi   , 1  
\end{array} \right)$.
Using $\vE$ to denote the matrix whose columns are the recourse vectors
$\vv_{\omega}$, equation \eqref{mpconstr} becomes $\vv = \vE \vD^{-1}
\left(\begin{array}{l} \vP \xi  , 1  \end{array} \right)$. This clearly implies
 that $\vv$  affinely depends on $\vP \xi$.

Let us now consider any affine policy $\vw + \vW \vP \xi$.  As shown above, the
recourse vector  $\vv$ provided by the multipolar approach is given by $\vE \vD^{-1}
\left(\begin{array}{l} \vP \xi  , 1  \end{array} \right)$. 
By taking $\vE = [\vW,\vw] \vD$, we get $\vv = \vw + \vW\vP  \xi$. In other words, any 
recourse policy that is affine in $\vP \xi$  can be obtained through the multipolar approach.  
\end{proof}
When $\vP = \vI$, we get the desired result below. 
\begin{corollary}\label{thm:general}
The affinely adjustable approach is a special case of the multipolar approach.
It corresponds to any set of $(\dim{\Xi}+1)$ affinely
independent poles, in multipolar robust optimization when $\vP=\vI$.
\end{corollary}
The following corollary is also immediate. 
\begin{corollary}\label{cor:general}
If the uncertainty set $\Xi$ is a simplex, then the
affinely adjustable robust counterpart   is equivalent to the fully adjustable robust counterpart
in the sense that their objective values are equal.
\end{corollary}
\begin{proof}
    Taking all the vertices of the simplex uncertainty set as the set of poles
    in multipolar robust approach leads to the optimum of FARC. By
    Corollary~\ref{thm:general}, this pole-set corresponds to affine adjustable
    approach, which completes the proof. 
\end{proof}
Corollary~\ref{cor:general} has been presented in \cite{BER2012} in the special
case of right-hand-side uncertainty, so we may treat the result here as an
alternative proof using the framework of multipolar approach.
\section{Analysis}\label{sec:analysis}
In this section, we first analyze the tractability of the multipolar robust
counterpart \ref{mpro}. Then, we show that the proposed framework can generate
a monotonic sequence converging to the fully adjustable robust value of FARC.
In fact, we will simultaneously generate a lower and an upper bound both
converging to the optimal value of FARC under some mild assumptions. 

\subsection{Tractability} \label{sec:tract}
In this section, we show that \ref{mpro} is computationally tractable. It can
be solved either by cut generation or using a compact reformulation.

First, a cutting plane algorithm for solving \ref{mpro} may be devised as
follows.  Assume that $\abs{\Omega}$ is finite and has a reasonable size. Given
a solution $(\vu,\vv)$, we have to check if there exists a pair of $\xi\in\Xi$
and $\vlambda^{\xi}\in\Lambda_{\xi}$ violating the constraints of \ref{mpro}.
This can be done by checking the sign of the optimum of each $i^{th}$ problem
\begin{subequations}\label{subpb}
\begin{align}
\max \limits_{\vlambda,\xi} \hspace{1em}& \vU_i \vu+\vV_i
\sum\limits_{\omega\in\Omega}\lambda^{\xi}_{\omega}\vv_{\omega} - b_i\\
\textrm{s.t.} \hspace{1em}&
\sum\limits_{\omega\in\Omega}\lambda^{\xi}_{\omega}\omega=\vP\xi,
\label{subpb:convex}\\
~~& \sum\limits_{\omega\in\Omega}\lambda^{\xi}_{\omega}=1, \label{subpb:unit}\\
~~& \lambda^{\xi}_{\omega}\ge 0,\;\;\omega\in\Omega,\\
~~& \xi\in\Xi, \label{subpb:uncertainty}
\end{align}
\end{subequations}
where $\vU_i$ and $\vV_i$ are the $i^{th}$ rows of $\vU$ and $\vV$. If it is positive, then constraint
\begin{equation}
\label{ineg}
  \vU_i \vu+\vV_i
  \sum\limits_{\omega\in\Omega}\hat\lambda^{\hat\xi}_{\omega}\vv_{\omega}\le b_i,
\end{equation}
needs to be added to the restricted problem, where $ (\hat\lambda^{\hat\xi},\;\hat\xi)$ solves \eqref{subpb}. Problem \eqref{subpb}
can generally be solved easily when $\Xi$ is polyhedral or ellipsoidal. In
these cases, by equivalence of separation and optimization \cite{GRO}, the
multipolar robust optimization  counterpart problem can also be solved in
polynomial time if the number of poles  $|\Omega|$ is polynomially bounded.

Second, we may solve \ref{mpro} by duality. It is sometimes
possible for several kinds of convex uncertainty sets to write a strong dual
of \eqref{subpb} leading to an extended reformulation of \ref{mpro}. This
holds for example if $\Xi$ is a polytope defined by a limited  number of
constraints, i.e., $\Xi:=\{\xi\equiv\lbrack \vU,\vb\rbrack: \vC\xi \leq \vd\}$,
where $\vC=\left[\vC_1,\dots, \vC_m\right]$, $\vC_i\in \rit^{n_d \times (n+1)}, d\in
\rit^{n_d}$ and $\xi$ is expressed as a column vector of size
$(n+1)\times m$. $\xi$ contains $m$ blocks of size $n+1$ vectors: the $i^{th}$
block contains $\vU_i^T$ followed by $b_i$. 
By strong duality, the constraints of the multipolar robust counterpart
\ref{mpro} w.r.t. $\Xi$ can be replaced with a polynomial number of
inequalities. For each $i$, the inequalities
$\vU_i\vu+\vV_i\sum\limits_{\omega\in\Omega}\lambda^{\xi}_{\omega}\vv_{\omega}\le b_i,
\xi\in\Xi,\;\vlambda^{\xi}\in\Lambda_{\xi}$ are replaced with 
\begin{equation}
\label{dual:polytope}
    \begin{split}
        & \vd^T\veta_i  +  \vV_i\vv_{\omega}-\omega^T\vsigma_i  \le 0, ~ \omega \in  \Omega, \\
             & \vC_j^T\eta_i - \vP^T_j\vsigma_i = \delta_{ij}\begin{pmatrix}
             \vu ,-1\end{pmatrix},  ~   j= 1,\dots,m,\\
             & \veta_i\in \rit^{n_d}_+,~\vsigma_i\in \rit^{n_0}, 
\end{split}
\end{equation}
where the shadow matrix $\vP=[\vP_1,\dots,\vP_j,\dots,\vP_m], \vP_j\in
\rit^{n_0\times{(n+1)}}, ~j=1,\dots,m$, $\delta_{ij}$ is Kronecker's delta function. \\

When $\Xi$ is ellipsoidal, i.e., $ \Xi:= \{\xi: \norm{\vF\xi}{2} \leq
1\}$, the multipolar robust counterpart can be represented by a second order
cone program. Then for each $i$, the $i^{th}$ constraint of \ref{mpro} is replaced with 
\begin{equation}\label{eq:counterpartball}
    \begin{split}
        \norm{\veta_i}{2}+\vV_i\vv_\omega - \omega^T\vsigma_i &\le 0 ,\\
        \vF^T\eta_i  - \vL_i &= \allzeros,\\
\veta_i \in \rit^{n_q},~\vsigma_i& \in \rit^{n_0}. 
\end{split}
\end{equation}
where $\vL_i=\begin{pmatrix}\vL_{i1},  \dots, \vL_{im}
\end{pmatrix}$, $ \vL_{ij} = \delta_{ij}\begin{pmatrix} \vu , -1\end{pmatrix} +
\vP_j^T\sigma_i, ~j=1,\dots,m$, $n_q$ is number of rows of matrix $\vF$.
For sake of completeness, a proof of \eqref{eq:counterpartball} is provided in Appendix. 

\subsection{Monotonicity}
We show in this section that the function $\Pi_{\Xi}(\vP,\cdot)$ is monotonic
w.r.t. a partial order defined on $\mathcal{F}_{\Xi_P}$ when the shadow matrix $\vP$ is fixed.

Given an uncertainty set $\Xi$, we now define a partial order over the collection
of its pole-sets $\mathcal{F}_{\Xi_P}$ denoted by
$\preceq_{\mathcal{F}_{\Xi_P}}$. We set members of $\mathcal{F}_{\Xi_P}$ ordered
by the inclusion of their convex hulls, i.e., for any $\Omega',\Omega \in
\mathcal{F}_{\Xi_P}$,
\begin{align}
    \Omega' \preceq_{\mathcal{F}_{\Xi_P}} \Omega \Longleftrightarrow  \conv
    \Omega' \subseteq \conv \Omega.
\end{align}
The next theorem emphasizes the fact that the function $\Pi_{\Xi}(\vP,\cdot)$ is monotonic
regarding the partial order $\preceq_{\mathcal{F}_{\Xi_P}}$ for each fixed $\vP\in
\rit^{n_0\times \dim(\Xi)}$. In other words, the multipolar value gets
smaller when $\Omega$ is smaller w.r.t. $\preceq_{\mathcal{F}_{\Xi_P}}$. 
\begin{theorem}  \label{theo:best}
Given $\vP\in \rit^{n_0\times \dim(\Xi)}$, for any $\Omega',\Omega \in
\mathcal{F}_{\Xi_P}$, if $\Omega' \preceq_{\mathcal{F}_{\Xi_P}} \Omega$, then
we have $\Pi_\Xi(\vP,\Omega') \le \Pi_\Xi(\vP,\Omega)$.
\end{theorem} 
\begin{proof}
    If $\left(\vu, ({\vv_{\omega}})_{\omega \in \Omega}
    \right)$ is an optimal solution of \ref{mpro}, then a feasible solution,
    when the set of poles is defined by $\Omega'$, is given as follows. Each
    $\omega' \in \Omega'$ writes as a convex combination of the poles of
    $\Omega$: $\omega' = \sum\limits_{\omega\in\Omega}
    \lambda^{\omega'}_{\omega} \omega$. Let $v_{\omega '} =
    \sum\limits_{\omega\in\Omega} \lambda^{\omega'}_{\omega}\vv_{\omega}$. The
    solution given by $\left(\vu, ({\vv_{\omega'}})_{\omega' \in \Omega'} \right)$
    is clearly feasible for \ref{mpro} w.r.t. the set of poles defined by
    $\Omega'$, which completes the proof.
\end{proof}
Theorem~\ref{theo:best} not only implies that the smaller the $ \Omega$  w.r.t.
$\preceq_{\mathcal{F}_{\Xi_P}}$,
the lower the multipolar robust value, but also implies that for a given $\vP$,
$\Pi_\Xi(\vP,\Omega)$ is minimum if $\Omega =\ext(\Xi_P)$. 

Given two pole-sets $\Omega, \Omega'\in \mathcal{F}_{\Xi_P}$,
Theorem~\ref{theo:best} also indicates that: first, for a fixed shadow matrix
$\vP$, if $\abs{\Omega}> \abs{\Omega'}$, then $\Pi_\Xi(\vP,\Omega)$ is not
necessarily less than $\Pi_\Xi(\vP,\Omega')$; second, the function
$\Pi_\Xi(\vP,\cdot)$ is not strictly monotonically increasing. For example, let
$S,S'\in \mathcal{F}_{\Xi_P}, S'\preceq_{\mathcal{F}_{\Xi_P}} S$ and their
convex hulls are simplices. By Theorem~\ref{thm1:general}, $\Pi_\Xi(\vP,S') =
\Pi_\Xi(\vP,S)$ while by Theorem~\ref{theo:best}, $\Pi_\Xi(\vP,S') \le
\Pi_\Xi(\vP,S)$, which illustrates the second point. Now take any pole-set
$\Omega$ whose  cardinality is strictly greater than $1+\dim(\Xi_P)$, such that
$S'\preceq_{\mathcal{F}_{\Xi_P}}\Omega\preceq_{\mathcal{F}_{\Xi_P}}S$, then
$\Pi_\Xi(\vP,S') = \Pi_\Xi(\vP,\Omega) = \Pi_\Xi(\vP,S)$, which illustrates the first
point.

Observe also that when $\vP = \vI$,  any pole-set whose convex hull contains $\Xi$ is  contained in a simplex. 
This immediately implies that  the optimal value of AARC  represents the worst that can be obtained by the multipolar approach.

\subsection{Convergence} 
\label{sec:converg}
The aim of this section is to show that under some mild assumptions, using the
multipolar framework, one can simultaneously compute a sequence of upper bounds
and a sequence of lower bounds converging to $\Pi_\Xi\left(\vI,\ext\left(\Xi\right)\right)$, the optimal
robust value of FARC. Throughout this section the shadow matrix is the identity matrix.
\begin{definition}
    Let $\Omega\in \mathcal{F}_{\Xi}$ be a pole-set of a non-empty set
    $\Xi$, the distance function between them is defined as $d\left(\Omega,\Xi\right) =
    \max\limits_{\omega\in \Omega}\min\limits_{\xi \in \Xi}  ~
    \norm{\omega-\xi}{2}.$
\end{definition}
The distance function is well-defined since $\Omega$ and $\Xi$ are closed and
bounded. It characterizes the furtherest distance between pole-set $\Omega$ and the
uncertainty set $\Xi$. This distance  is nothing other than the well-known Hausdorff distance. 

Let $\Omega \in \mathcal{F}_{\Xi}$ such that $d\left(\Omega,\Xi\right) = \epsilon$. 
For each $\omega\in \Omega $, we have $d\left(\Omega,\Xi\right)\le \epsilon$. 
Let $\vz_\omega$ be the projection of $\omega$ on  $\Xi$, i.e.,
\begin{align}\label{eq:condition1}
    \vz_{\omega}=\argmin\limits_{\vx\in \Xi} ~d\left(\omega,x\right),\hspace{2em}
    ~\ve_\omega = \omega -\vz_\omega, 
\end{align}
where $\norm{\ve_\omega}{2}\le \epsilon$. For each $\xi \in \Xi$, consider convex
combination coefficients $\left(\beta^{\xi}_{\omega}\right)$ such that 
\begin{align}
\label{eq:condition2}
  \xi=    \sum\limits_{\omega\in\Omega}\beta^{\xi}_{\omega}\omega, \hspace{2em} \mbox{and let}  \hspace{2em} 
 \vE = \sum\limits_{\omega\in\Omega}\beta^{\xi}_{\omega}\ve_\omega.
\end{align}

Let us add subscripts to avoid confusion: $\xi\equiv\lbrack
\vU_{\xi},\vb_{\xi}\rbrack$, $\vE \equiv\lbrack \vU_{E},\vb_{E}\rbrack$  and
$\vz_{\omega} \equiv\lbrack \vU_{\vz_\omega},\vb_{\vz_\omega}\rbrack$. We define the convex set 
 \begin{equation}
\Xi'_{z_{\Omega}} = \conv \left\{\vz_{\omega}: \omega \in \Omega \right\}.
\label{def:xi'}
\end{equation}
We obviously have $\Xi'_{z_{\Omega}}  \subseteq \Xi$.  Let
$(\vu^*,{(\vv^*_{\vz_{\omega}})}_{\omega \in \Omega})$ be the optimal solution of the
MRC problem related to $\Xi'_{z_{\Omega}}$. Due to the definition of
$\Xi'_{z_{\Omega}}$, MRC and FARC are equivalent. Moreover, from $\Xi'_{z_{\Omega}} \subseteq \Xi$, we get that  
\begin{align*}
\vc^T \vu^* =  \Pi_{\Xi'_{z_{\Omega}}}(\vI,\ext(\Xi'_{z_{\Omega}}))\le
\Pi_{\Xi}(\vI,\ext(\Xi)).
\end{align*}
We will also assume that there is a positive number $\mu$ such that
$\norm{(\vu^*,1)}{2} \leq \mu$. This assumption generally holds. For example, if
the cost vector $\vc$ is positive and variables $\vu$ are non-negative, then 
$\vc^T \vu^* \leq \Pi_{\Xi}(\vI,\ext(\Xi))$
implies that  $\norm{(\vu^*,1)}{2}$ is
upper-bounded. The number $\mu$ does not depend on $\epsilon$.
\begin{assumption}
\label{assump:bound}
There exists a constant number $\mu$ such that  $\norm{(\vu^*,1)}{2} \le
\mu$ for any $\Xi'_{z_{\Omega}} \subseteq \Xi$ and any optimal solution $\vu^*$ of
the FARC problem related to $\Xi'$.
\end{assumption}
 \begin{lemma} \label{lemma:1}
    Under Assumption~\ref{assump:bound}, for each $\xi \in \Xi$,  $\left(\vU_{\vE}
    \vu^* - \vb_{\vE} \right)$ is bounded from above by $\epsilon \mu \allones$, where
    $\allones$ is an all-ones vector. 
\end{lemma}
\begin{proof}
     The result follows from Cauchy-Schwartz
    inequality applied to each row of $\vE \equiv\lbrack \vU_{E},\vb_{E}\rbrack$.
\end{proof}
Let $\delta$ be a small positive number and let 
\begin{equation}
\Xi_\delta = \left\{\xi \equiv \lbrack \vU, \vb\rbrack: \exists \xi'\equiv \lbrack
\vU,\vb' \rbrack \in  \Xi, \norm{\vb-\vb'}{\infty} \le \delta\right\}.
\end{equation} Observe that if $\xi
\equiv \left[\vU, \vb\right] \in \Xi$, then $ \lbrack \vU,\vb - \delta \allones\rbrack \in
\Xi_\delta$. We will assume that for some small number $\delta$, the static
robust counterpart problem SRC is still solvable. 
\begin{assumption}
There exists a static robust solution $(\vu_{\delta}, \vv_{\delta})$ w.r.t. uncertainty set $\Xi_\delta$.
\label{assump:stat}
\end{assumption}
\begin{theorem}\label{thm:convergence}
 Under Assumptions~\ref{assump:bound} and~\ref{assump:stat}, for each
 pole-set $\Omega \in \mathcal{F}_\Xi$ such that $d(\Omega,\Xi)=\epsilon \le
 \frac{\delta}{\mu}$, we have
    \begin{align*} 
    \Pi_{\Xi}(\vI,\Omega) \le \left(1-\frac{\epsilon\mu}{\delta}\right)\vc^T\vu^*
    +\frac{\epsilon\mu}{\delta} \vc^T \vu_{\delta},
    \end{align*} 
    where $\vc^T \vu^*$ and $\vc^T
    \vu_\delta$ are respectively fully adjustable robust cost w.r.t. $\Xi'_{z_{\Omega}}$ and  static cost w.r.t. $\Xi_\delta$.
\end{theorem}
\begin{proof} 
Assume that the  optimal solution of FARC w.r.t. uncertainty set $\Xi'_{z_{\Omega}}$ 
is $\left(\vu^*, (\vv^*_{\vz_{\omega}})\right)$. Consider the solution
\begin{align}
    \hat \vu=\left(1-\frac{\epsilon\mu}{\delta}\right)\vu^* +\frac{\epsilon\mu}{\delta} \vu_\delta,
    \hspace{2em} \hat \vv_\omega =\left(1-\frac{\epsilon\mu}{\delta}\right)
    \vv^*_{\vz_\omega}+\frac{\epsilon\mu}{\delta} \vv_\delta.
\end{align}
Let us show that $(\hat \vu, \hat \vv_\omega)$ is a feasible solution of the
MRC problem related to $\Xi$ and $\Omega$. For  any $\xi\equiv
\left[\vU,\vb \right] \in \Xi$, by \eqref{eq:condition2}, one can write: 
\begin{align}
 \vU_{\xi}\hat \vu + \vV\sum\limits_{\omega \in \Omega} \beta^\xi_\omega \hat \vv_\omega
  = ~ & \vU_{\xi} \hat \vu + \left(1-\frac{\epsilon\mu}{\delta}\right) \vV\sum\limits_{\omega\in \Omega}
 \beta^\xi_\omega  \vv_{\vz_{\omega}}^*+\frac{\epsilon\mu}{\delta} \vV\vv_\delta \notag \\
 \le~ &\vU_{\xi} \hat \vu +\left(1-\frac{\epsilon\mu}{\delta}\right) \sum\limits_{\omega \in \Omega}
 \beta^\xi_\omega \left(\vb_{\vz_{\omega}}-\vU_{\vz_{\omega}
 }\vu^*\right)  \notag \\
 &\hspace{1.8em}+\frac{\epsilon\mu}{\delta} \left(\vb_{\xi}-\delta \allones -\vU_{\xi}\vu_\delta\right)
    \label{eq:equation1}\\
   =~ &  \vU_{\xi}\hat \vu  +\left(1-\frac{\epsilon\mu}{\delta}\right) 
    \left(\vb_{\xi} - \vb_{\vE}  +\vU_{\vE} \vu^*- \vU_{\xi}\vu^*
    \right) \notag \\
    &\hspace{1.8em}+\frac{\epsilon\mu}{\delta} \left(\vb_{\xi} -\delta \allones -\vU_{\xi}\vu_\delta\right) \label{eq:equation2} \\
   =~ & \left(1-\frac{\epsilon\mu}{\delta}\right)\left(\vU_{\vE} \vu^* - \vb_{\vE}\right)-{\epsilon\mu}
  \allones +\vb_{\xi}  \notag\\
  \le ~& \left(1-\frac{\epsilon\mu}{\delta}\right)  \epsilon \mu \allones  -{\epsilon\mu}
 \allones +\vb_{\xi} \label{eq:equation22}\\
   =  ~&  \vb_{\xi}  -\frac{\epsilon^2\mu^2}{\delta} \allones \notag\\
 \le ~& \vb_{\xi}, \notag
\end{align}
where \eqref{eq:equation1} follows from the fact that $\left(\vu^*,\vv_{\vz_{\omega}}^*\right)$ satisfies 
constraint $\vU\vu+\vV\vv\le \vb$ for $\vz_{\omega}=\left[\vU_{\vz_{\omega}},\vb_{\vz_{\omega}}\right]^T$ and the 
static solution $\left(\vu_\delta, \vv_\delta\right)$ satisfies $\vU_{\xi}\vu+\vV\vv\le \vb_{\xi}-\delta \allones$, 
\eqref{eq:equation2} follows from \eqref{eq:condition1} and
\eqref{eq:condition2}, and \eqref{eq:equation22} is due to Lemma~\ref{lemma:1}.

The robust cost incurred by $\left(\hat \vu,\left(\hat \vv_\omega\right)\right)$ is
$\left(1-\frac{\epsilon\mu}{\delta}\right)\vc^T\vu^* +\frac{\epsilon\mu}{\delta}
\vc^T\vu_{\delta}$ and is an upper bound of the optimum of \ref{mpro}.
\end{proof}
\begin{corollary}\label{cor:twosequence}
Given any sequence of pole-sets $\Omega_i \in \mathcal{F}_{\Xi}$ such that
$\lim\limits_{i \rightarrow \infty}  d\left(\Omega_i, \Xi\right) =0$, then under
Assumptions~\ref{assump:bound} and~\ref{assump:stat},
\begin{align*}
\Pi_\Xi\left(\vI,\Omega_i\right) \geq \Pi_\Xi\left(\vI,\ext\left(\Xi\right)\right)
~\textrm{and}~
\lim\limits_{i \rightarrow \infty}  \Pi_\Xi\left(\vI,\Omega_i\right)= \Pi_\Xi\left(\vI,\ext\left(\Xi\right)\right).
\end{align*}
Moreover, the corresponding sequence of sets $\Xi'_{z_{\Omega_i}}$ defined in
\eqref{def:xi'} satisfies 
\begin{align*}
\Pi_{\Xi'_{z_{\Omega_i}}}\left(\vI, \ext\left(\Xi'_{z_{\Omega_i}}\right)\right) \leq \Pi_\Xi\left(\vI,\ext\left(\Xi\right)\right)
~\textrm{and}~
 \lim\limits_{i \rightarrow \infty}  \Pi_{\Xi'_{z_{\Omega_i}}}\left(\vI, \ext\left(\Xi'_{z_{\Omega_i}}\right)\right) = \Pi_\Xi\left(\vI,\ext\left(\Xi\right)\right).
 \end{align*}
\end{corollary}
\begin{proof}
Let $\epsilon_i = d\left(\Omega_i, \Xi\right), \forall i$. From Theorem~\ref{thm:convergence}, we have
\begin{align*}
\Pi_\Xi\left(\vI,\Omega_i\right)  \le \left(1-\frac{\epsilon_i\mu}{\delta}\right)\Pi_{\Xi'_{z_{\Omega_i}}}
\left(\vI,\ext\left(\Xi'_{z_{\Omega_i}}\right)\right)  + \frac{\epsilon_i\mu}{\delta} \vc^T \vu_{\delta}, ~
\forall i,
\end{align*}
and we know that 
\begin{align*}
\Pi_{\Xi'_{z_{\Omega_i}}} \left(\vI,\ext\left(\Xi'_{z_{\Omega_i}}\right)\right)  \le \Pi_\Xi\left(\vI,\ext\left(\Xi\right)\right)
\le\Pi_\Xi\left(\vI,\Omega_i\right).
\end{align*}
Consequently, 
\begin{align*}
\lim\limits_{i \rightarrow \infty}  \Pi_{\Xi'_{z_{\Omega_i}}}
\left(\vI,\ext\left(\Xi'_{z_{\Omega_i}}\right)\right) = \Pi_\Xi\left(\vI,\ext\left(\Xi\right)\right) ~\textrm{and} ~
\lim\limits_{i \rightarrow \infty}  \Pi_\Xi\left(\vI,\Omega_i\right)= \Pi_\Xi\left(\vI,\ext\left(\Xi\right)\right)
\end{align*}
hold in the limit at the same time. 
\end{proof}

\section{An illustrative example} \label{sec:example}
To illustrate the multipolar concept, we present a simple example, which had
been previously studied in \cite{CHEN} and is as follows: 
\begin{equation}
\label{eq_example}
\begin{array}{lll}
    \min \hspace{1em} &u \\
    \textrm{s.t.} \hspace{1em} & \forall \xi \in \rit^{n},\norm{\xi}{1} \leq 1, &
    \exists \vv,   v_i \geq \xi_i, v_i \geq - \xi_i, i=1,\ldots,n \\
&  &u \geq \sum\limits_{i=1}^{n} v_i.
\end{array}
\end{equation}
Observe that $u$ is here the unique first-stage (non adjustable) variable. On
the other hand $v_i$, for each $i=1,\ldots,n$, are second-stage (adjustable)
variables.  The uncertainty set is given by $\Xi \equiv \left\{\xi \in
\rit^{n},\norm{\xi}{1} \leq 1 \right\}$.

As noticed in \cite{CHEN}, an optimal fully adjustable solution is given by $u = 1$ and
$v_i = \norm{\xi_i}{1}$, whereas the optimal affinely adjustable solution requires that
$u = n$. In other words,  the affine approach does not lead to any improvement
compared to the static approach. 

Following the paradigm of multipolar approach in Section \ref{sec:concept}, let
us take $\vP \xi = \left(\xi_1,\ldots,\xi_{n_0}\right)$, where $n_0\in
\mathbb{N}, n_0 \leq n$. In other words, the shadow matrix $\vP$ limits the
dimension of $\Xi$ to $n_0$ by leaving $\xi_i$ as they are for $i\leq n_0$ and
disregarding the other components for $i>n_0$. Let $\Omega \subseteq
\rit^{n_0}$  be the set of poles containing for $i=1,\dots,n_0,$ vectors
${\phi}^i = \left(0,\ldots,0,1,0,\ldots,0\right)$ and $\overline{\phi}^i  =
-\phi^i$, whose components are $0$ except the $i^{th}$ component. Hence
$\Omega$ contains $2n_0$ poles and $\Xi_P = \conv \Omega$.

Given any $\xi \subseteq \Xi$, let $\lambda_{{\phi}^i }$  and
$\lambda_{\overline{\phi}^i }$ be the convex combination coefficients such that
$\vP \xi = \sum\limits_{i=1}^{n_0}  \left(\lambda_{{\phi}^i }  {\phi}^i  +
\lambda_{\overline{\phi}^i } \overline{\phi}^i\right) $. The equation can be
transformed to  $\vP \xi = \sum\limits_{i=1}^{n_0} {\phi}^i
\big(\lambda_{{\phi}^i } \allowbreak - \lambda_{\overline{\phi}^i }\big) $; thus these
coefficients should satisfy the equations $\lambda_{{\phi}^i } -
\lambda_{\overline{\phi}^i } = \xi_i$ for $1 \leq i \leq n_0$. Let
$\vv_{{\phi}^i }$ (resp. $\vv_{\overline{\phi}^i }$) be the recourse vector
associated with pole ${{\phi}^i }$ (resp.  ${\overline{\phi}^i }$).  These
vectors belong to $\rit^{n}$.

In the considered example, inequalities \eqref{ineg} are equivalent to the
following set of inequalities: 
\begin{subequations}
\begin{align}
\sum\limits_{i=1}^{n_0}  \left(\lambda_{{\phi}^i }    \vv_{{\phi}^i}  +
\lambda_{\overline{\phi}^i }  \vv_{\overline{\phi}^i}\right) \geq
\left(\abs{\xi_1},\abs{\xi_2},\ldots,\abs{\xi_n}\right)  \label{v_eq}, \\
u \geq \norm{ \sum\limits_{i=1}^{n_0}  \left(\lambda_{{\phi}^i }    \vv_{{\phi}^i}  +
\lambda_{\overline{\phi}^i }   \vv_{\overline{\phi}^i}\right)}{1}.\label{u_eq}
\end{align}
\end{subequations}

Let us take $ \vv_{{\phi}^i}  = \vv_{\overline{\phi}^i} =
(0,\ldots,0,1,0,\ldots,0,1,\ldots,1)$, where the first $n_0$ components are
$0$ except the $i^{th}$ component, which is equal to $1$, while the last $(n -
n_0)$ components are equal to $1$.

Observe that  the last $\left(n - n_0\right)$ components of the vector
$\sum\limits_{i=1}^{n_0}  \left(\lambda_{{\phi}^i }    \vv_{{\phi}^i}  +
\lambda_{\overline{\phi}^i }   \vv_{\overline{\phi}^i}\right) $ are equal to $1$.
Moreover, for $1 \leq i \leq n_0$, we have $ \lambda_{{\phi}^i }  +
\lambda_{\overline{\phi}^i } \geq |\lambda_{{\phi}^i } -
\lambda_{\overline{\phi}^i }  | = |\xi_i |$. This clearly implies that
inequalities \eqref{v_eq} are satisfied. In addition, inequality \eqref{u_eq}
leads to $u \geq  \sum\limits_{i=1}^{n_0}  \left(\lambda_{{\phi}^i }  +
\lambda_{\overline{\phi}^i }  \right) \left(1 + n - n_0\right)$, where $\left(1 + n - n_0\right)$ is the
$L^1$ norm of each recourse vector $ \vv_{{\phi}^i}$.  Consequently, $u \geq 1 +
n - n_0 $.  Since we are minimizing $u$, we  get $u = 1 + n - n_0$. The cost
decreases when $n_0$ increases. When $n_0$ is equal to $1$,  we get a static
solution, while the optimal fully adjustable solution is obtained when $n_0 = n$.
Finally, taking $1<n_0<n$, we obtain a compromise between the simplicity of the
static approach and the efficiency of the fully adjustable solution. As mentioned
earlier, such a compromise cannot be obtained for this example with the
affinely adjustable approach.

Consider now a slightly changed example with the uncertainty set being the
non-polyhedral set defined by $\Xi := \{\xi \in \rit^{n}: \norm{\xi}{2} \leq 1\}$. 
 The rest of the
problem remains as in \eqref{eq_example}; thus the new problem can be
formulated as follows:

\begin{equation}
\label{eq_example2}
\begin{array}{lll}
    \min  \hspace{1em}&u \\
    \textrm{s.t.} \hspace{1em} & \forall \xi \in \rit^{n},||\xi||_2 \leq 1, &
    \exists \vv,   v_i \geq \xi_i, v_i \geq - \xi_i, i=1,\ldots,n,\\
     &  &u \geq  \sum\limits_{i=1}^{n} v_i.
\end{array}
\end{equation}
Observe that since the new uncertainty set $\Xi$ contains the previous one
based on $L^1$ norm, the optimal value of \eqref{eq_example2} is greater than
or equal to that of \eqref{eq_example}. Optimal solutions based on either the static
approach or the affine approach still incur a cost of $n$ while the optimal
fully adjustable solution has a cost of $\sqrt{n}$.
Let us now consider the multipolar approach, where $\vP$ is still defined by $P
\xi = (\xi_1,\ldots,\xi_{n_0})$. Let us choose the following set of poles:
$\Omega= \{\sqrt{n_0} {\phi}^i\}_{i=1}^{i=n_0} \cup \{\sqrt{n_0}
\overline{\phi}^i \}_{i=1}^{i=n_0}$. One can easily show that $ \Xi_P \subseteq
\conv(\Omega)$. Moreover, by taking $ \vv_{\sqrt{n_0} {\phi}^i}  =
\vv_{\sqrt{n_0}\overline{\phi}^i } = (0,\dots, \allowbreak 0, \sqrt{n_0}, 0, \dots, 0,
1, \dots, 1)$, where the first $n_0$ components are $0$ except the $i^{th}$ component, which
is equal to $\sqrt{n_0}$, while the last $n - n_0$ components are equal to $1$,
we get a solution of the multipolar robust counterpart with $u = \sqrt{n_0} + n
- n_0$.   Similarly to the previous case, when $n_0$ is equal to $1$,  we get a
static solution, while the optimal fully adjustable solution is obtained when $n_0 = n$.
%
\section{The construction of pole-sets}\label{sec:poles}

We know from Section \ref{sec:analysis} that the multipolar robust value converges
to a fully adjustable robust value when the distance between $\Omega$ and
$\Xi_P$ gets close to $0$, and $\vP = \vI$. We also proved the
monotonicity of multipolar robust value w.r.t. the inclusion of $\conv
\Omega$. Therefore, the objective of this section is to find a pole-set $\Omega
\in \mathcal{F}_P$ as close to $\Xi_P$ as possible, while minimizing the number
of poles. This is clearly related to the theory of approximation of convex sets
by polytopes. 

A considerable amount of work has been done in this area. A recent survey of
relevant results is given in \cite{BRON}. It is proved in \cite{BRONS,DUD} that
given a convex body $\Xi_P \in \rit^{n_0}$, there exists a polytope $ F_n \in
\rit^{n_0}$ having $n$ vertices containing $\Xi_P$ such that $d_H(
\Xi_P ,  F_n) \leq \frac{k\left(\Xi_P\right)}{n^{2/(n_0-1)}}$ where $d_H$
denotes the Hausdorff distance and $k(\Xi_P)$ is a constant only depending on
$\Xi_P$. More precise approximations are obtained in dimension $2$, where we
can ensure the existence of $F_n \subset\rit^2$ such that $d_H( \Xi_P ,
F_n) \leq \frac{l}{2 n} \sin{\frac{\pi}{n}}$ where $l$ is the length of the
boundary of $\Xi_P$. Moreover, if the boundary of $\Xi_P$ is two-times smooth,
then an explicit asymptotic result is known about the distance between $\Xi_P$
and the set of circumscribed polytopes having $n$ vertices: the closest
 polytope
$F_n$ satisfies $d_H( \Xi_P ,   F_n) \sim
\frac{k\left(\Xi_P\right)}{n^{2/(n_0-1)}}$ where $k(\Xi_P)$ is a constant
depending on $n_0$ and the Gaussian curvature of the boundary of $\Xi_P$
\cite{BOR}.  

The monotonicity of multipolar robust values w.r.t. pole-sets might suggest
using minimum volume circumscribed polytopes. Considering the Nikodym
distance (related to volumes) instead of the Hausdorff distance, the same kind
of results can be obtained \cite{BRON}. One might be interested in a minimum
volume simplex containing a convex set $\Xi_P$. We know for example that
if $\Xi_P$ is the hypercube $H_{n_0}$, then a minimum volume circumscribed
simplex has a volume equal to $\frac{{n_0}^{n_0} }{{n_0}!}$ \cite{LASSA}.
If $\Xi_P$ is the unit ball, then a minimum volume simplex containing
the ball is a regular simplex whose volume is
$\frac{n_0^{n_0/2}(n_0+1)^{(n_0+1)/2}}{n_0!}$ \cite{BRON} and whose
dihedral angle is $\arccos(\frac{1}{n_0})$
\cite{parks2002}. 
It is also known that a minimum volume simplex enclosing $\Xi_P$ satisfies the
centroid property: the centroid of each facet of this simplex should be in
$\Xi_P$ \cite{KLEE}. A polynomial-time algorithm to find such a
minimum volume simplex enclosing a set of points in $\mathbb R^3$ is given
in \cite {ZHOU}. However, it is generally unknown how to solve the
problem in higher dimensions \cite{GRIT}.

As observed by \cite{BRON}, most constructive algorithms were  generally
proposed for low dimensional cases ($2$ or $3$). For more general cases,
constructive algorithms of circumscribed polytopes such as the algorithm of
\cite{KAM} are generally based on the addition of inequalities without
controlling the number of vertices of the circumscribed polytope. This can
hardly accommodate the need of multipolar framework since we want to control the
complexity of MRC by limiting the number of poles.  

Note also that we are required to construct the pole-set  of $\Xi_P$ in a
reasonable time. Algorithms checking whether each extreme point of $\Xi_P$
belongs to the convex hull of $\Omega$ fail to work, since the number of
extreme points of a polytope can be exponential or even infinite.

The rest of this section is organized as follows. First, we describe a general
algorithm to construct a simplex enclosing $\Xi_P$. The resulting simplex is
guaranteed to be smallest in the sense that it cannot be shrinked. 
Then, a project-and-cut based tightening procedure is proposed to construct
pole-sets that are closer to $\Xi_P$.  
 
\subsection{Generation of a circumscribed simplex}
\label{sec:genercir}
In this section, we describe a general algorithm for the construction of a
circumscribed simplex of $\Xi_P$. Specifically, we first randomly generate a
set of $(n_0+1)$ affinely independent points, whose convex hull
forms a simplex $S$. Then, we compute the best homothetic
transformation of $S$ such that the resulting simplex contains $\Xi_P$. 

We denote by the $\omega^{(i)}, i =1,\dots, (n_0 +1)$ the $(n_0+1)$ affinely
independent points. Then the $n_0$-simplex set can be expressed as $\left\{\vx: ~
\vD\vlambda = (\vx , 1), \vlambda \geq \allzeros\right\}$, where 
\begin{align*}
    \vD=\begin{pmatrix} \omega^{(1)}_1 &\hdots & \omega^{(n_0+1)}_1 \\ \vdots &\vdots &\vdots
\\ \omega^{(1)}_{n_0} &\hdots &\omega^{(n_0+1)}_{n_0} \\ 1 &\hdots& 1 \end{pmatrix}. 
\end{align*}
Since the $(n_0+1)$ points are affinely independent, matrix $\vD$ is
invertible; therefore, $\lambda_j,~j = 1,\dots,(n_0+1)$, can be expressed as a
affine function of  $\vx$; the coefficients of the
affine function $\lambda_i$ are the components of the $i$th row of $\vD^{-1}$, i.e.,
\begin{align}\label{eq:lambda}
    \lambda_i(x)= \sum\limits_{j=1}^{n_0}l_{ij}x_j + l_{i(n_0+1)}, ~ i = 1,\dots,
    (n_0+1).
\end{align}
Note that $\lambda_i(x)\geq 0, ~ i = 1,\dots, (n_0+1)$ iff $\vx$ belongs to the $n$-simplex.

Let  $\vSigma_{\sigma}, \vT_\vt$ be the associate matrices for the operations of scaling with
factor $\sigma > 0$ and translation $\vt \in \rit^{n_0}$. Thus the associated
matrix $\vD_{\sigma,t}$ of the simplex with homothetic transformation on the simplex $S$ is 
$\vD_{\sigma,t} = \vT_{\vt} \vSigma_{\sigma} \vD$,
where
\begin{align*}
\vT_\vt = \begin{pmatrix} \vI_{n_0} &\vt \\ 0 &1 \end{pmatrix}, ~
\vSigma_{\sigma}  = \begin{pmatrix} \sigma \vI_{n_0}  & 0\\ 0& 1\end{pmatrix}.
\end{align*}
Its corresponding inverse is then
\begin{align*}
\vD_{\sigma,\vt}^{-1} & = \vD^{-1} 
\begin{pmatrix} \frac{1}{\sigma}I_{n_0}  & 0\\ 0& 1\end{pmatrix}\begin{pmatrix}
\vI_{n_0} &-\vt \\ 0 &1 \end{pmatrix}\\
                                       &= \begin{pmatrix}
    \frac{1}{\sigma} l_{11} &\hdots    & \frac{1}{\sigma}l_{1 n_0}   &l_{1(n_0+1)} - \frac{1}{\sigma}\sum\limits_{k=1}^{n_0}
t_k l_{1k} \\ \vdots &\vdots &\vdots  & \vdots \\ \frac{1}{\sigma} l_{n_0 1} &\hdots & \frac{1}{\sigma}l_{n_0 n_0}
                     &l_{n_0 (n_0+1)} -\frac{1}{\sigma}\sum\limits_{k=1}^{n_0} t_k l_{n_0 k} \\
 \frac{1}{\sigma}l_{(n_0+1)1} &\hdots  & \frac{1}{\sigma} l_{(n_0+1) n_0} & l_{(n_0+1)(n_0+1)} -\frac{1}{\sigma}\sum\limits_{k=1}^{n_0} t_k l_{(n_0+1)k} \end{pmatrix}. 
\end{align*}
Let $\sigma^*$ be the smallest scaling factor $\sigma$ such that a translate of
$\sigma S$ contains $\Xi_P$. The translate used when $\sigma = \sigma^*$ is
denoted by $\vt^*$.
\begin{theorem}\label{thm:general_translate}
$\sigma^*$ and  $\vt^*$ are given by: 
   $\vt^* = \sum\limits_{i=1}^{n_0+1} z_i\omega^{(i)}$ and 
    $\sigma^* =  -\sum\limits_{i=1}^{n_0+1} z_i$,
where for each $i=1,\ldots,(n_0+1)$,
$z_i = \min\{\sum\limits_{j=1}^{n_0} l_{ij}x_j: ~\vx\in \Xi_{P}\}.$
\end{theorem}
\begin{proof}
Assume that the homothetic copy of $S$ given by  $\sigma S + \vt$  contains
$\Xi_P$. Then the coefficients $\lambda_i(x)$ defined in  \eqref{eq:lambda}
should be nonnegative for any point $\vx \in \Xi_P$. 
Considering the matrix $\vD_{\sigma,t}^{-1}$ defined above and computing the
minimum values of $\lambda_i(x), i=1,\ldots,(n_0+1)$,  
we get 
\begin{align*}
l_{i(n_0+1)}-\frac{1}{\sigma}\sum\limits_{k=1}^{n_0} t_k l_{ik}+  \frac{1}{\sigma}
z_i \ge 0, ~ i= 1, \dots, (n_0+1), 
\end{align*}
For ease of notation, we express this as
  $  l_{i(n_0+1)}' + z_i \geq 0, ~ i= 1, \dots, (n_0+1), $
where $l_{i(n_0+1)}' = \sigma l_{i(n_0+1)}- \sum\limits_{k=1}^{n_0} t_k l_{ik} , ~ i= 1,
\dots, (n_0+1)$. \\
Since the matrix $(l)_{i, j =1,\dots,(n_0+1)}$  is the inverse of matrix $\vD$, we
have
\begin{equation}
 \label{eq:inverse}
\sum\limits_{i=1}^{n_0+1} l_{ij} = 0, ~j= 1, \dots, n_0~\mbox{ and } 
    \sum\limits_{i=1}^{n_0+1} l_{i(n_0+1)} = 1. 
\end{equation}

Summing all $l'_{i (n_0+1)}$, we get
\begin{equation}
\label{eq:sigma}
\sigma=\sum\limits_{i=1}^{n_0+1} l'_{i (n_0+1)} \ge -\sum\limits_{i=1}^{n_0+1} z_i.\end{equation}

Observe that having $l'_{i (n_0+1)}$ (and $\sigma^*$ as a consequence), we can
get the translate $t$ through the linear system  $ \sum\limits_{k=1}^{n_0} t_k l_{ik} = \sigma
l_{i(n_0+1)}-l_{i(n_0+1)}' , ~ i= 1, \dots, (n_0+1)$. Multiplying by $\vD$, we get that
$(t_1,...,t_{n_0}, 0) = \sigma (0,...,0,1) - \vD (l'_{1 (n_0+1)},..., l'_{(n_0+1)  (n_0+1)})$
which leads to $\vt = \sum\limits_{i=1}^{n_0+1} z_i\omega^{(i)}$.

According to \eqref{eq:sigma}, the smallest $\sigma^*$ is $-\sum\limits_{i=1}^{n_0+1} z_i$.
We should however check if $\sigma^* \geq 0$. This holds since by considering any $x \in \Xi_P$, one can write  that  
\begin{align*}
\sum\limits_{i=1}^{n_0+1} z_i  \leq &  \sum\limits_{i=1}^{n_0+1} \sum_{j=1}^{n_0} l_{ij} x_j \\
                                               = & \sum_{j=1}^{n_0} x_j  \sum\limits_{i=1}^{n_0+1} l_{ij} \\
                                               = & 0, 
\end{align*}
where the last equality is based on \eqref{eq:inverse}.  

Since the matrix $\vD_{\sigma,t}^{-1}$ does not exist when $\sigma^* = 0$, we
have to study this special case. It is clear that $\sigma^* = 0$   if and only
if $\Xi_P$ is a single point. Observe that in this case, we necessarily have
    $ \sum\limits_{i=1}^{n_0+1} z_i = 0$ since $ z_i = \sum_{j=1}^{n_0} l_{ij}
    x_j$, where $\vx$ is the single point of $\Xi_P$. Then formula 
$\sigma^* = - \sum\limits_{i=1}^{n_0+1} z_i$ is still valid and $\vt^* =  \vx =
\sum\limits_{i=1}^{n_0+1} z_i\omega^{(i)}$  also occurs.   
\end{proof}

Note that values of $z_i, i=1,\dots, (n_0+1)$ defined in
Theorem~\ref{thm:general_translate} can easily be computed for any $\Xi_P$
since we only have to minimize a linear function over a convex set.  

As a special case, pole-sets of a hypercube are of great use in multipolar
robust approach. First, hypercubes are one of the most common uncertainty sets
in many applications. Second, general box sets of the form
$\left\{\vx: \vx\in \left[\vl,\vu\right]^{n_0} \subseteq \rit^{n_0}\right\}$ are simply affine
transformations of a hypercube, so the pole-sets of a hypercube also apply to
boxes with some simple transformations.    
\begin{corollary}\label{cor:hypercube}
If $\Xi_P$ is a hypercube, then $\sigma^*$ and  $\vt^*$ are given by: 
$$
     \vt^* =  \sum\limits_{i=1}^{n+1} \sum\limits_{j=1}^{n_0}
      \min\{0, l_{ij}\}\omega^{(i)} \mbox{ and }
    \sigma^* =  \frac{1}{2}\sum\limits_{i=1}^{n_0+1} \sum\limits_{j=1}^{n_0} \abs{l_{ij}}. 
$$
\end{corollary}
\begin{proof}
If $\Xi_P$ is a hypercube, by Theorem~\ref{thm:general_translate}, we have $
z_i = \sum\limits_{j=1}^{n_0} \min\{0, l_{ij}\}, ~i=1,\dots,n_0$.
By \eqref{eq:inverse}, we have
$
-\sum\limits_{i=1}^{n_0+1} \sum\limits_{j=1}^{n_0}
\min\{0,l_{ij}\} = \frac{1}{2}\sum\limits_{i=1}^{n_0+1} \sum\limits_{j=1}^{n_0}
\abs{l_{ij}}$,
which completes the proof.
\end{proof}
According to Corollary~\ref{cor:hypercube}, we have a closed formula for the
homothetic translation for a  $n_0$-simplex $S$  containing the $n_0$-hypercube, i.e,
\begin{align}
    \vx_{\sigma, t} = \frac{1}{2}
    \sum\limits_{i=1}^{n_0+1}\sum\limits_{j=1}^{n_0}\abs{l_{ij}} \vx +
    \sum\limits_{i=1}^{n_0+1} \omega^{(i)}\sum\limits_{j=1}^{n_0} \min\left\{0,
    l_{ij}\right\} ,   ~\forall \vx\in S.
\end{align}
Note that the value $\sigma^*$ presented in Corollary~\ref{cor:hypercube} has
been  given in \cite{Nevskii2011} but the proof here is much simpler.  

To sum up the foregoing, we present a general algorithm for the generation of a circumscribed simplex as follows.
\begin{enumerate}
\item Generate $(n_0+1)$ affinely independent points $(\omega^{(i)})_{i=1}^{i=n_0+1}$.
\item Compute $\sigma^*$ and $\vt^*$  by Theorem ~\ref{thm:general_translate} and output the
$\sigma^*\omega^{(i)} + \vt^*, i=1,\dots, (n_0+1)$.
\end{enumerate}

\subsection{A tightening procedure} \label{sec:tight}
In this section, we propose a general procedure to construct pole-sets of good
quality by tightening a given pole-set.

The procedure is the following: among the vertices of $\Omega$ select the farthest
one in $L^2$ sense from $\Xi_P$ and compute the projection of this vertex on
$\Xi_P$.  Then we consider the hyperplane separating this vertex from $\Xi_P$
(containing the projection) and compute the extreme points of the intersection of this
hyperplane with $\conv \Omega$. These extreme points are added to $\Omega$ while the
vertex that has been projected is removed from $\Omega$. 
Figure \ref{fig:tightening} illustrates a tightening procedure of a 3-D simplex
covering $H_3$. The procedure is repeated until the cardinality of $\Omega$
reaches some fixed upper bound. 
\begin{figure}[t]
\tdplotsetmaincoords{70}{120}
\begin{tikzpicture}[tdplot_main_coords,scale=1.5]
\def\BigSide{1}
\def\SmallSide{1}
\pgfmathsetmacro{\CalcSide}{\BigSide-\SmallSide}
\tdplotsetcoord{P}{sqrt(3)*\BigSide}{55}{45}
\draw[dashed]
    (0,0,0) -- (Px)
    (0,0,0) -- (Py)
    (0,0,0) -- (Pz);
\draw
  (Pxz) -- (P) -- (Pxy) -- (Px) -- (Pxz) -- (Pz) -- (Pyz) -- (P); 
\draw
 (Pyz) -- (Py) -- (Pxy);
\draw[very thick] (1.7535,-0.3684,0)--(0.9141,0.7368,3.0)--(0.3324,2.6316,0)--cycle;
\draw[very thick,dotted] (-1.2465,-0.3684,0)-- (1.7535,-0.3684,0);
\draw[very thick,dotted] (-1.2465,-0.3684,0)--(0.9141,0.7368,3.0);
\draw[very thick,dotted] (-1.2465,-0.3684,0)--(0.3324,2.6316,0);
\draw[very thin,red,->](.9141,0.7368,3.0)--(.9141,0.7368,1);
\draw[fill=black] (.9141,0.7368,3.0) circle(1pt) node[left] {$\omega^0$};
\draw[fill=gray] (.9141,0.7368,1) circle(1pt);
\draw[fill=blue, opacity=0.3] (-0.5263,0,1)--(1.4737,0,1)--(0.5263,2,1)--cycle;
\draw[fill=black] (-0.5263,0,1) circle(1pt) node[left,text width=5mm]{$\omega^1$};
\draw[fill=black] (1.4737,0,1) circle(1pt) node[left] {$\omega^2$};
\draw[fill=black] (0.5263,2,1) circle(1pt) node[right] {$\omega^3$};
\end{tikzpicture}
\hspace{2cm}
\tdplotsetmaincoords{70}{115}
\begin{tikzpicture}[tdplot_main_coords,scale=1.5]
\def\BigSide{1}
\def\SmallSide{1}
\pgfmathsetmacro{\CalcSide}{\BigSide-\SmallSide}
\tdplotsetcoord{P}{sqrt(3)*\BigSide}{55}{45}
\draw[dashed]
    (0,0,0) -- (Px)
    (0,0,0) -- (Py)
    (0,0,0) -- (Pz);
\draw
  (Pxz) -- (P) -- (Pxy) -- (Px) -- (Pxz) -- (Pz) -- (Pyz) -- (P); 
\draw
 (Pyz) -- (Py) -- (Pxy);
\draw[very thick,dotted] (-1.2465,-0.3684,0)-- (1.7535,-0.3684,0);
\draw[very thick,dotted] (-1.2465,-0.3684,0)--(0.3324,2.6316,0);
\draw[very thick,dotted] (-1.2465,-0.3684,0)--(-0.5263,0,1);
\draw[very thick](-0.5263,0,1)--(1.4737,0,1)--(0.5263,2,1)--cycle;
\draw[very thick] (1.4737,0,1)--(0.5263,2,1)--(0.3324,2.6316,0)--(1.7535,-0.3684,0) --cycle;
\draw[fill=black] (-0.5263,0,1) circle(1pt) node[left,text width=5mm]{$\omega^1$};
\draw[fill=black] (1.4737,0,1) circle(1pt) node[left] {$\omega^2$};
\draw[fill=black] (0.5263,2,1) circle(1pt) node[right] {$\omega^3$};
\end{tikzpicture}
\caption{(Left) Pole $\omega^0$ is replaced with poles $(\omega^i)_{i=1}^{3}$.
(Right) The updated convex hull of the new set of poles.}
\label{fig:tightening}
\end{figure}
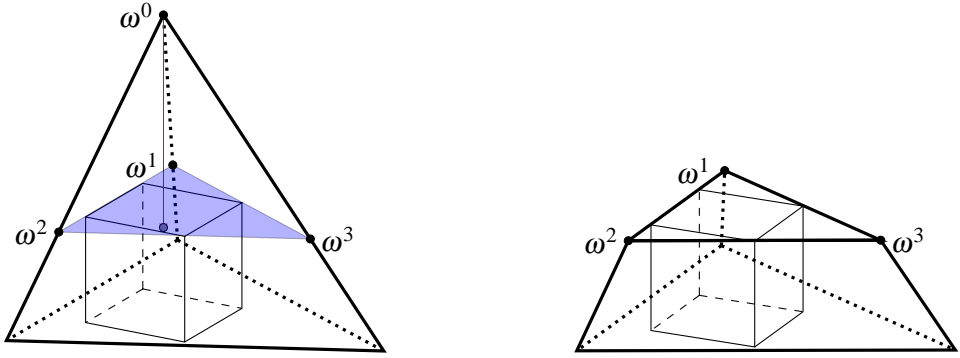
Details are given below:
\begin{enumerate}
    \item Assume $\Omega = \left\{ \omega^{(k)}, k \in I \right\}$.  For each $k \in I$,
    compute the distance between $\omega^{(k)}$  and $\Xi_P$.  Let $\vz^k$ be the
    projection of $\omega^{(k)}$ on $\Xi_P$. $\vz^k$ can be usually expressed in a
    closed form.  For example, in the ball case, we have
    $\vz^k=\frac{\omega^{(k)}}{\norm{\omega^{(k)}}{2}}$, while in the hypercube
    case we get $\vz^k_i =   \omega^{(k)}_i$ if $\omega^{(k)}_i\in [0,1]$, $z^k_i
    =  1$ if $\omega^{(k)}_i \geq 1$ and $z^k_i =  0$ if not. 
 The distance between $\omega^{(k)}$ and $\Xi_P$ is then given by
$\norm{\omega^{(k)} - \vz^k}{2}$. Let $\omega^{(k_0)}$ be the vertex of $\Omega$
maximizing the distance from
$\Xi_P$:  $$\omega^{(k_0)} = \arg\max\limits_{\omega^{(k)} \in \Omega}
\norm{\omega^{(k)} - \vz^k}{2}.$$
\item  Let $\valpha =\omega^{(k_0)} - \vz^{k_0}$ and let
$B\left(\omega^{(k_0)},\norm{\valpha}{2}\right)$  be the ball of radius $\norm{\valpha}{2}$
centered at $\omega^{(k_0)}$.  Since
$B\left(\omega^{(k_0)},\norm{\valpha}{2}\right) \cap
\Xi_P = \{\vz^{k_0}\}$ and $B\left(\omega^{(k_0)},\norm{\valpha}{2}\right)$ and $\Xi_P$ are convex,
there is a hyperplane separating them. This hyperplane, denoted by
$h(\omega^{(k_0)})$, is here uniquely defined since it contains $\vz^{k_0}$ and
is orthogonal to $\valpha$. It is then given by  $ h\left(\omega^{(k_0)}\right) = \left\{ \vx:
(\vx - \vz^{k_0})^T \valpha  = 0\right\}$. We use
$h^-\left(\omega^{(k_0)}\right)=\left\{ \vx: (\vx -
\vz^{k_0})^T\valpha  <  0\right\}$  and $h^+\left(\omega^{(k_0)}\right)=\left\{
    \vx:  (\vx - \vz^{k_0})^T \valpha \geq 0 \right \}$ to respectively denote
    the inner and outer half spaces. 

  \item Now partition the vertices $\left(\omega^{(k)}\right)_{k \in I}$ into
  two disjoint sets: $\Omega^-$ and $\Omega^+$, where $\Omega^-=\left\{\omega^{(k)}: \omega^{(k)} \in
      h^-\left(\omega^{(k_0)} \right)\right\} $
      and $\Omega^+ = \left\{\omega^{(k)}: \omega^{(k)} \in
    h^+\left(\omega^{(k_0)}\right)\right\}$.  
Then consider the set of  vertices $ \Omega'$  obtained as intersections
between the hyperplane $h(\omega^{(k_0)})$ and the set of lines
$(\omega^{(i)},\omega^{(j)})$ where  $\omega^{(i)}  \in \Omega^+$ and
$\omega^{(j)} \in  \Omega^-$:  $ \Omega'= \bigcup_{\omega^{(i)}  \in \Omega^+,
\omega^{(j)} \in  \Omega^-} h(\omega^{(k_0)}) \cap
    (\omega^{(i)},\omega^{(j)})$. 
The number of such intersections is of course less than $|\Omega^-| \times
|\Omega^+|$. Also note that we need to remove redundant points from  $\Omega'$
if they are convex combinations of other points of $\Omega'$. Finally update  $\Omega$ by
deleting $\Omega^+$ and adding $\Omega'$: $\Omega = \Omega^- \cup \Omega'$.  If
cardinality of $\Omega$ is still under a prescribed upper bound, the procedure
is repeated.  
\end{enumerate}
To conclude this section, we might add that it is sometimes more efficient to
start with a pole-set having more than $(n_0 + 1)$ poles. Assume, for example,
that $\Xi_P$ is the unit ball  $\left\{\vx\in \rit^{n_0}: \norm{\vx}{2} \le 1
\right\}$. Then one can consider a \emph{2n-pole-set} where poles are the $2
n_0$ extreme points of $ \left\{\vx\in \rit^{n_0}: \norm{\vx}{1} \le
\sqrt{n_0}\right\}$. Of course, 2n-pole-sets can also be easily generated for
many other convex sets. 

\section{A numerical example: the lobbying problem}\label{sec:problem}

Let us consider a lobbying problem where a set of voters (for example,
legislators  or members of regulatory agencies) have to take some decisions.
The opinion of each voter depends on the opinion of some authorities.
Authority's opinions are generally uncertain. A lobby would like to ensure
that an important decision will be unanimously approved by all voters. The
lobby will spend some effort (energy, money, etc.) to convince each voter,
while the total lobbying budget is minimized. Assume that there are $m$ voters
and $n$ authorities. The opinion of voter $i$ is given by $\sum_{j=1}^{n}
Q_{ij} \xi_j$ where $Q_{ij}$ is an estimated number belonging to $[-1,1]$ and
$\xi_j$ represents the uncertain opinion of  authority $j$.  If $Q_{ij}$ is
close to $1$, then $j$ has a big impact on $i$, while $Q_{ij} = 0$ means that
$i$ does not care about $j$, while a negative value of $Q_{ij}$ can be
interpreted as a negative effect (i.e., when $j$ recommends something, $i$ is
inclined to have an opposite opinion). We assume here that the lobby would
be satisfied if $\sum_{j=1}^{n} Q_{ij} \xi_j \leq 0$ for each voter. Since
this might not occur for some voters, some effort modeled here by $v_{i}(\xi)$
can be made by the lobby to convince them. The total effort is quantified by
$\sum_{i=1}^m  r_i  v_{i}(\xi)$ where $r_i$ is a unit effort price
corresponding with voter $i$. The problem can be formulated as follows,
\begin{equation}\label{eq:problem}
\begin{split}
    \min\hspace{1em}& u         \\
    \textrm{s.t.}\hspace{1em} & \sum_{i=1}^{m}  r_i v_{i}(\xi)  \le u, ~\xi \in \Xi, \\
    & \vQ\xi \leq \vv(\xi), ~\xi\in \Xi,\\
    & \vv(\xi)  \ge \allzeros, ~\xi \in \Xi,
\end{split}
\end{equation}
where $\vQ\in [-\allones,\allones]^{m\times n}$, $\Xi$ is the convex
uncertainty set and $u$ is the budget that has to be secured by the lobby.
The lobby problem is related to opinion dynamics in social networks (see
\cite{ACE} and the references therein). Notice that interactions between voters
are also possible since the set of authorities might include the set of voters
as a subset. 

To illustrate the multipolar robust approach, we consider here two different
uncertainty sets: the hypercube $H_n$ and a unit volume ball $B_n$. The numbers
$r_i$ are assumed to be equal to $1$. Specializing \eqref{dual:polytope} to
$H_n$, we get the following formulation for MRC associated with a shadow matrix
$\vP$ and a feasible pole-set $\Omega\in \mathcal{F}_{P{H_n}}$,
\begin{equation}\label{eq:prob_hypercube}
    \begin{split}
\Pi_{H_n}(\vP,\Omega)=\min\limits_{u,\vv_\omega}\hspace{1em}& u\\
    \textrm{s.t.}\hspace{1em} & \sum\limits_{i=1}^{m} \vv_\omega^i + \sum\limits_{j=1}^n\sigma_j\omega_j +
        \sum\limits_{j=1}^n \beta_j \le u, ~  \omega \in \Omega,\\
        &  \sum\limits_{j=1}^{n}\eta^T_{ij}\omega_{j} +
        \sum\limits_{j=1}^{n}\alpha_{ij}- v_{\omega}^{i}\leq 0, ~\omega \in \Omega, i=1,\dots, m, \\
        & \sum\limits_{j=1}^{n}\mu_{ij} - \vV_{\omega}^i -
        \sum\limits_{j=1}^n\delta_{ij}\omega_j \leq
        0,  ~\omega \in \Omega, i=1,\dots, m, \\
        & \vbeta + \vP^T\vsigma\ge \allzeros,  \\ 
        & \alpha_{ij} + \vP^T\veta_{i}\ge Q_{ij}, ~ i= 1,\dots, m, j=1,\dots, n, \\
        & \vmu_{i}-\vP^T\tau_{i} \ge \allzeros, ~i=1,\dots,m, \\
        & \alpha_{ij}, \beta_j, \mu_{ij}\ge 0, ~ i= 1,\dots, m, j=1,\dots, n,\\
        & \vsigma, \veta_i,\tau_i \in \rit^{n_0}, ~i=1,\dots,m.
    \end{split}
\end{equation}
As said above, we also consider the case where the uncertainty set is  a unit
volume ball $B_n$, whose center is
$\bar \xi = (\frac{1}{2},\hdots,\frac{1}{2},\hdots,\frac{1}{2})$ and radius
$\rho=(\frac{\Gamma(n/2+1)}{\pi^{n/2}})^{\frac{1}{n}}$.  According to
\eqref{eq:counterpartball}, the dual of the multipolar robust
counterpart w.r.t. $B_n$, writes
\begin{equation}\label{eq:prob_ball}
\begin{split}
 \Pi_{B_n}(\vP,\Omega) = \min\limits_{u,\vv_\omega}\hspace{1em}& u \\
    \textrm{s.t.}\hspace{1em}
    & \rho \norm{\vP^T\sigma}{2} +\allones^T\vv_\omega -\omega^T\vsigma \le u, ~~\omega \in \Omega, \\
    & \rho \norm{\vQ_i^T+\vP^T\veta_i}{2} -\vv_\omega^i
    -\omega^T\veta_i + \vQ_i\bar \xi\le   0, ~~\omega \in \Omega, i=1,\dots,m, \\ 
    & \rho \norm{\vP^T\vmu_i}{2}  -\vv_\omega^i -\omega^T\vmu_i \le   0,   ~\omega \in \Omega, i=1,\dots, m, \\ 
    & \vsigma , \veta_i, \vmu_i\in \rit^{n_0},  ~ i=1,\dots, m,
\end{split}
\end{equation}
where $\Omega \in \mathcal{F}_{P{B_n}}$.

As stated in Section \ref{sec:special}, the fully adjustable robust value
w.r.t. hypercube $H_n$ can be achieved by simply taking $\vP=\vI$ and
$\Omega=\ext(H_n)$. 
The problem looks more complex in the  ball case since the number of extreme
points is infinite.  Assume again that $\vP$ is identity.
Given $\xi\in B_n$, the optimal solution of $v_i$ is $\max\{0, \vQ_i\xi\}$ for
each $i=1,\dots,m$. Denote by $\mathcal{P}$ the power set of the index set
\{1,\dots,m \}.
We partition the ball $B_n$ into a family of disjoint subsets by a set valued
mapping $S: \mathcal{P}\mapsto 2^{B_n}$, i.e., for each $J\in \mathcal{P}$,
\begin{align*}
    S(J):=\left\{\xi\in B_n: \vQ_i \xi \geq 0, ~i\in J, ~\vQ_j\xi \le 0, ~j\in
    \bar J\right\},
\end{align*}
where $\bar J=\{1,\dots,m\}\setminus J$. Therefore, the fully adjustable robust program writes 
\begin{align}\label{eq:FARC}
    \Pi^*_{B_n}= \max\limits_{J \in \mathcal{P}:S(J)\neq
    \emptyset}~\max\limits_{\xi \in S(J)} \sum\limits_{i\in J}\vQ_i\xi.
\end{align}
Notice that \eqref{eq:FARC} takes an exponential number (in the number of
constraints $m$) of seconder order cone programs to obtain the fully adjustable robust value
$\Pi^*_{B_n}$. We show that \eqref{eq:FARC} is
equivalent to a much simpler problem.
\begin{lemma}\label{lemma:FARC_simple}
   Program \eqref{eq:FARC} is equivalent to
\begin{align} \label{eq:FARC_simple}
\Pi^*_{B_n} = \max\limits_{J \in \mathcal{P}}\left\{ \rho \norm{\sum\limits_{i\in
J}\vQ_i}{2}+\sum\limits_{i\in J}\vQ_i\bar \xi\right\}.
\end{align}
\end{lemma}
\begin{proof}
Observe first that $\max\limits_{J \in \mathcal{P}}\left\{ \rho
\norm{\sum\limits_{i\in J}\vQ_i}{2}+\sum\limits_{i\in J}\vQ_i\bar \xi\right\}$ is
an upper bound of  $\Pi^*_{B_n}$ since  it is obtained by relaxing the
constraints $\xi \in S(J)$. 

We show that it is also a lower bound of $\Pi^*_{B_n}$. Let  $J_{max}\in
\mathcal{P}$ be a subset for which the maximum is achieved:
$$J_{max} = \arg\max\limits_{J \in \mathcal{P}}\left\{ \rho
\norm{\sum\limits_{i\in J}\vQ_i}{2}+\sum\limits_{i\in J}\vQ_i\bar \xi\right\}.$$
Take $\hat\xi  \in B_n$ such that 
$\hat \xi=  \bar \xi+\frac{\rho \sum\limits_{i\in
J_{max}}\vQ_i}{\norm{\sum\limits_{i\in J_{max}}\vQ_i}{2} }$.
Let $K\in \mathcal{P}$ such that
\begin{align*}
\vQ_i \hat \xi \ge 0, &~ i\in K,\\ 
\vQ_i \hat \xi \le 0, &~ i\in \{1,\dots,m\}\setminus K.
\end{align*}
We have that 
\begin{align*}
 \rho \norm{\sum\limits_{i \in J_{max}} \vQ_i}{2}+\sum\limits_{i\in J_{max}}\vQ_i\bar \xi
 &= \sum\limits_{i\in J_{max}} \vQ_i\hat\xi \\
 &\le \sum\limits_{i \in K} \vQ_i\hat \xi \\
 &\le \max_{\xi\in S(K)}\sum\limits_{i\in K} \vQ_i\xi  \\
 &\le \Pi^*_{B_n}
\end{align*}
where the first equality follows from the choice of $\hat \xi$. The second
inequality is due to $J_{max}\cap K \subseteq K$,
while the third inequality is from  the fact that $\hat\xi$ belongs to $S(K)$ by
    the definition of $K$. 
 The last inequality is a direct consequence of \eqref{eq:FARC}. 
\end{proof}
Although problem \eqref{eq:FARC_simple} is easier than problem \eqref{eq:FARC},
it is still computationally costly when the number of constraints $m$ is large.
Notice that \eqref{eq:FARC_simple} can also be seen as a integer quadratic
program  that can be approximated by  semidefinite programming  and solved
using  standard quadratic programming tools.  We will not elaborate more on
this since this falls out of the scope of the paper. 
\subsection{Numerical experiments}\label{sec:numerical}
 The problem instances are randomly generated following the rules below.
\begin{enumerate}
    \item Let $m\in\{10,20,30, 40, 50\}$, $n\in\{5, 9,10,12,15,20,30\}$.
\item Generate the components of $\vQ$ uniformly over $[-\allones,\allones]$. 
    \item We build four different sizes of pole-sets for each considered
    hypercube by the circumscribed simplex generation algorithm and the
    tightening procedure described in Section~\ref{sec:poles}. As a result, for
    a hypercube $H_n$, $(\Omega_i)$ is a monotonic sequence w.r.t. the set
    inclusion of their convex hulls, i.e., $\Omega_i \preceq_{H_n} \Omega_j$
    for all $i>j$. Table \ref{tab:poles} displays the cardinality of different
        pole-sets of $H_n$. The number of vertices of $H_n$ is also provided in
        the last column.  
    \begin{table}[htb]
    \begin{center}
    \caption{The pole-sets of hypercubes}   
    \label{tab:poles}
\smallskip
    \begin{tabularx}{\textwidth}{@{}lp{4em}l*4{>{\raggedleft\arraybackslash}X}@{}}
\toprule\noalign{\smallskip}
    Hypercube        & $\abs{\Omega_0}$ & $\abs{\Omega_1}$ & $\abs{\Omega_2}$ &
    $\abs{\Omega_3}$ & \#ext.\\
\noalign{\smallskip} \midrule\noalign{\smallskip}
     $H_9 $  & 10 & 32 & 162 & 387 & 512\\
\noalign{\smallskip}
    $H_{10}$ & 11 & 36 & 112 & 322 & 1,024\\
\noalign{\smallskip}
    $H_{12}$ & 13 & 44 & 144 & 449 & 4,046\\
\noalign{\smallskip}
    $H_{15}$ & 16 & 56 & 192 & 353 & 32,768\\
\noalign{\smallskip}
    $H_{20}$ & 21 & 76 & 144 & 514 & 1,048,576\\
\noalign{\smallskip}
    $H_{30}$ & 31 & 60 & 116 & 432 & 1,073,741,824\\
\noalign{\smallskip} \bottomrule
    \end{tabularx}
    \end{center}
\end{table}    
\item As an illustration of multipolar robust approach for smooth convex
uncertainty sets, we generate pole-sets of ball $B_n$ as well. In Table
\ref{tab:pole_ball}, $\Omega_0$ denotes a first pole-set whose convex hull is a
simplex, while $\Omega_1$ is the $2n-$pole-set defined at the end of  Section
\ref{sec:poles}. Starting from $\Omega_1$ and applying the tightening
procedure, we get the pole-sets $\Omega_i, i=2,3,4$ as outputs. The cardinality of
pole-sets are shown in Table \ref{tab:pole_ball}.
    \begin{table}[htb]
    \caption{The pole-sets of the unit volume ball}
    \label{tab:pole_ball}
\smallskip
    \begin{tabularx}{\textwidth}{XXXXXX}
\toprule\noalign{\smallskip}
    Unit ball& $\abs{\Omega_0}$ & $\abs{\Omega_1}$ & $\abs{\Omega_2}$
    &$\abs{\Omega_3}$ &$\abs{\Omega_4}$  \\
\noalign{\smallskip} \midrule\noalign{\smallskip}
    $B_5 $   & 6  & 10 & 118 & 218 & 308  \\
\noalign{\smallskip}
    $B_9 $   & 10 & 18 & 62  & 152 & 352 \\
\noalign{\smallskip}
    $B_{10}$ & 11 & 20 & 72  & 132 & 374 \\
\noalign{\smallskip}
    $B_{12}$ & 13 & 24 & 88  & 164 & 478 \\
\noalign{\smallskip} \bottomrule
    \end{tabularx}
\end{table}    
\end{enumerate}
The pole-sets had been readily generated before the solution procedure.
Compact formulations \eqref{eq:prob_hypercube} and \eqref{eq:prob_ball} are
modeled by YALMIP~\cite{yalmip} and all the problem instances are solved by the Linux version
of CPLEX 12.5 with default settings on a Dell E6400 laptop with Intel Core(TM)2
Duo CPU clocked at 2.53 GHz and with 4 GB of RAM. We evaluate our multipolar
approach from different measuring: the impact of pole-sets, 
the impact of the shadow matrix $\vP$, and the benefit of adaptability. 
\subsubsection{The influence of pole-sets}
Recall that multipolar robust approach closes to some extend the gap between
affine robust value and the fully adjustable value. To test the impact of
pole-sets, we fix $\vP=\vI$.  For relatively lower dimensional cases, we report the
multipolar robust values w.r.t. different pole-sets, and compute the
percentage of the closed gap induced by the multipolar robust approach 
$  \frac{\Pi_\Xi(\Omega_0)-\Pi_\Xi(\Omega_i)}{ \Pi_\Xi(\Omega_0)
    -\Pi_\Xi^*}\times 100$.
For higher dimensional cases, where FARC can hardly be solved in a reasonable
time, we report the multipolar robust values w.r.t. different pole-sets. 

Results related to hypercubes are presented in  Table \ref{tab:benefit}.  The
closed gap percentages are given within parentheses. Overall, these
results appear encouraging as indicated by the closed gaps. Observe that when
the uncertainty set is fixed, the multipolar robust values in general get lower as the pole-set $\Omega$ gets smaller.
\begin{table}[ht]
    \begin{center}
    \caption{The multipolar robust values with different pole-sets (hypercube uncertainty
    sets)}
\label{tab:benefit}
\smallskip
\footnotesize
 \begin{tabularx}{\textwidth}{@{}ll*5{>{\raggedleft\arraybackslash}X}@{}}
     \toprule\noalign{\smallskip}
(m,n) & static & affine/$\Pi_{H_n}(\Omega_0)$ & $\Pi_{H_n}(\Omega_1)$ &
$\Pi_{H_n}(\Omega_2)$ & $\Pi_{H_n}(\Omega_3)$ & $\Pi_{H_n}^*$\\
\noalign{\smallskip} \hline\noalign{\smallskip}
(10,9)  & 24.84 & 12.42 & 12.18(9.45 )& 10.55(73.62) & 10.16(88.98)& 9.88  \\
\noalign{\smallskip}
(10,10) & 25.50 & 12.75 & 11.53(58.65)& 10.96(86.06) & 10.70(98.56)& 10.67  \\
\noalign{\smallskip}
(10,12) & 30.66 & 15.33 & 14.63(35.18)& 13.71(81.41) & 13.43(95.48)& 13.34  \\
\noalign{\smallskip}
(20,9)  & 50.75 & 25.37 & 23.82(46.69)& 22.08(99.10)& 22.06(99.70)& 22.05  \\
\noalign{\smallskip}
(20,10) & 50.88 & 25.44 & 23.56(16.95)& 20.74(42.38)& 18.58(61.86)& 14.35  \\
\noalign{\smallskip}
(20,12) & 59.79 & 29.89 & 27.54(23.64)& 25.40(45.17)& 23.58(63.48)& 19.95  \\
\noalign{\smallskip} \hline\noalign{\smallskip}
(10,15) & 35.81  & 17.90 & 16.91 & 15.98 & 15.29 & - \\
\noalign{\smallskip}
(10,20) & 50.88  & 25.44 & 24.82 & 24.35 & 23.33 & - \\
\noalign{\smallskip}
(10,30) & 64.32  & 32.16 & 31.70 & 31.14 & 30.22 & - \\
\noalign{\smallskip}
(20,15) & 82.49  & 41.25 & 39.36 & 36.20 & 34.87 & - \\
\noalign{\smallskip}
(20,20) & 99.28  & 49.64 & 47.17 & 45.66 & 40.80 & - \\
\noalign{\smallskip}
(20,30) & 157.20 & 78.60 & 77.82 & 76.83 & 73.77 & - \\
\noalign{\smallskip} \bottomrule
\end{tabularx}
\end{center}
\end{table}
Also, we report the computing time for higher dimensional instances associated
with hypercube uncertainty set in Table \ref{tab:cpu}. While the computational
time compared with the affine robust approach scales in magnitude, the
complexity of multipolar robust approach 
is controlled by the choice of pole-sets. In particular, in higher dimensional
cases, where fully adjustable robust values are difficult to obtain, the robust
multipolar approach brings lower cost (compared with affine approach) in a
reasonable time. 
\begin{table}[htbp]
\caption{Computing time (in seconds)}
\label{tab:cpu}
\smallskip
\begin{tabularx}{\textwidth}{@{}l*5{>{\centering\arraybackslash}X}@{}}
\toprule\noalign{\smallskip} 
(m,n) & static & affine/$\Pi_{H_n}(\Omega_0)$ & $\Pi_{H_n}(\Omega_1)$ &
$\Pi_{H_n}(\Omega_2)$ & $\Pi_{H_n}(\Omega_3)$ \\
\noalign{\smallskip} \hline\noalign{\smallskip}
(10,15)     & 0.00 & 0.01 & 0.20 & 1.49  & 6.31     \\
\noalign{\smallskip}
(10,20)     & 0.00 & 0.03 & 0.87 & 2.66  & 27.54  \\
\noalign{\smallskip}
(10,30)     & 0.00 & 0.04 & 1.15 & 4.53  & 39.12  \\
\noalign{\smallskip}
(20,15)     & 0.00 & 0.03 & 0.68 & 10.43 & 34.98   \\
\noalign{\smallskip}
(20,20)     & 0.00 & 0.07 & 4.65 & 16.41 & 152.79  \\
\noalign{\smallskip}
(20,30)     & 0.00 & 0.14 & 3.91 & 15.40 & 220.48 \\
\noalign{\smallskip} \bottomrule
\end{tabularx}
\end{table}

A sequence of lower bounds can also be generated as stated in
Corollary~\ref{cor:twosequence} of Section~\ref{sec:converg}. All we need to do is
to generate a sequence of $(\Gamma_i)_{i=0}^{i=3}$ by projecting the pole-sets
$(\Omega_i)_{i=0}^{i=3}$ onto the surface of hypercubes. The obtained lower bounds are denoted  by $\Pi_{H_n}(\Gamma_i)$. Note that
$\conv \Omega'\subseteq \Omega$ does not necessarily lead  to $\conv \Gamma'
\subseteq \Gamma$ or $\conv \Gamma' \supseteq \Gamma$.  
Thus it may happen that
$\Pi_{H_n}(\Gamma_i)\ge \Pi_{H_n}(\Gamma_{i+1})$. The results are summarized in Table
\ref{tab:lowerH}, where the best lower bound for each problem instance is marked
in bold. 
\begin{table}[h]
\caption{Lower bounds related to hypercubes} 
\label{tab:lowerH}
\smallskip
\begin{tabularx}{\textwidth}{@{}l*5{>{\centering\arraybackslash}X}@{}}
\toprule\noalign{\smallskip} 
(m,n) &$\Pi_{H_n}(\Gamma_0)$  & $\Pi_{H_n}(\Gamma_1)$ & $\Pi_{H_n}(\Gamma_2)$ &
$\Pi_{H_n}(\Gamma_3)$ & $\Pi_{H_n}^*$\\
\noalign{\smallskip} \hline\noalign{\smallskip}
(10,9)  & 6.88        & 8.18        & 9.52       & {\bf 9.65}  & 9.88\\
\noalign{\smallskip}
(10,10) & 7.11        & 8.12        & {\bf 9.62} & 8.34        & 10.67 \\
\noalign{\smallskip}
(10,12) & {\bf 10.50} & 8.88        & 9.48       & 9.57        & 13.34 \\
\noalign{\smallskip}
(20,9)  & 20.08       & 16.05       & 18.70      & {\bf 21.98} & 22.05 \\
\noalign{\smallskip}
(20,10) & 11.88       & 11.88       & 12.44      & {\bf 12.92} & 14.35 \\
\noalign{\smallskip}
(20,12) & 14.73       & 16.65       & 16.96      & {\bf 19.07} & 19.95 \\
\noalign{\smallskip}
(10,15) & 7.62        & 7.63        & 8.40       & {\bf 8.40}  & - \\
\noalign{\smallskip}
(10,20) & 7.86        & 10.20       & 10.20      & {\bf 10.36} & - \\
\noalign{\smallskip}
(10,30) & 7.11        & 9.03        & 9.03       & {\bf 9.79}  & - \\
\noalign{\smallskip}
(20,15) & 23.49       & {\bf 25.86} & 23.56      & 23.56       & - \\
\noalign{\smallskip}
(20,20) & 15.23       & 16.08       & 16.08      & {\bf 19.61} & - \\
\noalign{\smallskip}
(20,30) & 30.53       & 31.12       & 32.53      & {\bf 33.81} & - \\
\noalign{\smallskip} \bottomrule
\end{tabularx}
\end{table}

Let us now focus on the ball case. Remember that  FARC is intractable in general.
However, as shown in Lemma~\ref{lemma:FARC_simple}, we can compute the optimum
of FARC by solving problem \eqref{eq:FARC_simple} when $m$ is small.  We report
the robust values obtained by solving multipolar robust counterpart with
different pole-sets in Table \ref{tab:ball}. The results may
indicate the following. First, the approximate robust values associated with balls
appear lower than those associated with hypercubes although the
volume and symmetric center of balls and hypercubes are the same. Second,
the closed gaps by multipolar robust approach on robust problems with ball
uncertainty sets might be less significant than that with hypercube uncertainty
sets. As might be expected, larger poles-sets are required for balls compared to hypercubes.  Third, despite the 
limitations, multipolar approach closes around $30\%$ of the optimality gap. In
particular, it appears compelling when the number of constraints are large,
while the dimension of the uncertainty set is small. 
\begin{table}[ht]
    \begin{center}
    \caption{The multipolar robust values with different pole-sets (ball)}
\label{tab:ball}
\smallskip
\footnotesize
\begin{tabularx}{\textwidth}{@{}ll*6{>{\raggedleft\arraybackslash}X}@{}}
\toprule\noalign{\smallskip}
(m,n)&static &$\Pi_{B_n}(\Omega_0)$ & $\Pi_{B_n}(\Omega_1)$ & $\Pi_{B_n}(\Omega_2)$ & $\Pi_{B_n}(\Omega_3)$ &
$\Pi_{B_n}(\Omega_4)$&$\Pi_{B_n}^*$\\
\noalign{\smallskip}\hline\noalign{\smallskip}
(10,9)  & 17.27 & 9.43  & 9.28(11.11)  & 9.21(16.30)  & 9.11(23.70)  &
9.01(31.11)  & 8.08 \\
\noalign{\smallskip}
(10,10) & 16.09 & 9.21  & 8.98(23.00)  & 8.95(26.00)  & 8.94(27.00)  &
8.89(32.00)  & 8.21 \\
\noalign{\smallskip}
(10,12) & 19.64 & 10.93 & 10.76(20.99) & 10.76(20.99) & 10.74(23.46) &
10.70(28.40) & 10.12\\
\noalign{\smallskip}
(20,9)  & 35.87 & 19.86 & 19.54(20.92) & 19.51(22.88) & 19.45(26.80) &
19.36(32.68) & 18.33 \\
\noalign{\smallskip}
(20,10) & 33.12 & 17.19 & 16.42(18.08) & 15.62(36.85) & 15.57(38.03) &
15.36(42.96) & 12.93 \\
\noalign{\smallskip}
(20,12) & 39.85 & 20.93 & 20.05(17.25) & 19.96(19.02) & 19.90(20.20) &
19.81(21.96) & 15.83\\
\noalign{\smallskip}
(30,5)  & 19.51 & 11.00 & 10.39 & 9.22  & 9.14  & 9.03  & - \\
\noalign{\smallskip}
(40,5)  & 37.57 & 21.63 & 21.11 & 20.71 & 20.63 & 20.55 & - \\
\noalign{\smallskip}
(50,5)  & 38.14 & 20.94 & 20.06 & 19.33 & 19.14 & 19.02 & - \\
\noalign{\smallskip}\bottomrule
\end{tabularx}
    \end{center}
\end{table}

The lower bounds obtained in the ball case are reported in Table
\ref{tab:lowerB}. Interestingly, the observed sequences of lower bounds
associated with ball $B_n$ are monotonically increasing and their best bounds
in general are close to the fully adjustable robust value.   
\begin{table}[h]
\caption{The lower bounds in the ball case} 
\label{tab:lowerB}
\smallskip
\begin{tabularx}{\textwidth}{@{}l*6{>{\centering\arraybackslash}X}@{}}
\toprule\noalign{\smallskip}
(m,n) &$\Pi_{B_n}(\Gamma_0)$  & $\Pi_{B_n}(\Gamma_1)$ & $\Pi_{B_n}(\Gamma_2)$ &
$\Pi_{B_n}(\Gamma_3)$ & $\Pi_{B_n}(\Gamma_4)$ &$\Pi_{B_n}^*$\\
\noalign{\smallskip} \hline\noalign{\smallskip}
(10,9)  & 6.60  & 6.70  & 6.70  & 6.71  & {\bf 7.67} & 8.08\\
\noalign{\smallskip} 
(10,10) & 6.07  & 6.28  & 6.93  & 7.44  & {\bf 7.60 }& 8.21\\
\noalign{\smallskip} 
(10,12) & 7.26  & 7.51  & 8.01  & 8.08  & {\bf 8.44 }& 10.12\\
\noalign{\smallskip} 
(20,9)  & 16.04 & 16.17 & 16.89 & 16.89 & {\bf 16.89} &  18.33\\
\noalign{\smallskip} 
(20,10) & 12.03 & 12.03 & 12.03 & 12.03 & {\bf 12.03}  & 12.93\\
\noalign{\smallskip} 
(20,12) & 12.35 & 13.41 & 13.67 & 13.67 & {\bf 13.67}  & 15.83\\
\noalign{\smallskip} 
(30,5) & 7.48  & 7.48  & 8.11  & 8.11  & {\bf 8.34 }& - \\
\noalign{\smallskip} 
(40,5) & 16.54 & 17.31 & 19.27 & 19.27 & {\bf 19.87}& -  \\
\noalign{\smallskip} 
(50,5) & 15.53 & 15.53 & 17.22 & 17.22 & {\bf 17.80} & - \\
\noalign{\smallskip} \bottomrule
\end{tabularx}
\end{table}
\subsubsection{The impact of the shadow matrix}
To investigate the impact of the shadow matrix on the robust value of the
robust problem, we conduct some experiments on problem instances w.r.t.
hypercube uncertainty sets.  
The shadow matrices considered here are simply projection matrices on lower
subspaces. The results are displayed in Table \ref{tab:projection}, where the
uncertainty set is a hypercube and several projections are considered (on
$H_5$, $H_7$, $H_{10}$, and $H_{12}$). The pole-set considered is the set of extreme points of the projected set.
As might be expected, the robust value
decreases as more information is employed in MRC. 
\begin{table}[ht]
\caption{Impact of the  shadow matrix}
\label{tab:projection}
\smallskip
\begin{tabularx}{\textwidth}{@{}l*4{>{\centering\arraybackslash}X}@{}}
\toprule\noalign{\smallskip}
(m,n)    & $H_5$ & $H_7$ & $H_{10}$ & $H_{12}$ \\
\noalign{\smallskip} \hline\noalign{\smallskip}
(10,30)  & 56.95 & 51.50 & 45.59 & 43.29 \\
\noalign{\smallskip}
(10,50)  & 109.84 & 105.68 & 100.69 & 96.11 \\
\noalign{\smallskip}
(10,70)  & 162.14 & 159.52 & 156.12 & 154.53 \\
\noalign{\smallskip}
(10,100) & 246.98 & 244.13 & 239.34 & 237.43 \\
\noalign{\smallskip}
(20,30)   & 133.35 & 126.59 & 117.63 & 111.53\\
\noalign{\smallskip}
(20,50)  & 233.73 & 224.82 & 213.04 & 203.89 \\
\noalign{\smallskip}
(20,70)  & 345.56 & 340.41 & 328.49 & 319.79 \\
\noalign{\smallskip}
(20,100) & 462.51 & 453.26 & 439.71 & 431.39\\
\bottomrule
\end{tabularx}
\end{table}
\subsubsection{The benefit of adaptability}
To illustrate the concept of \emph{benefit of adaptability} in the framework of
multipolar robust approach, we compute the multipolar robust values of problem
\eqref{eq:problem} with different proportions of adjustable variables $\vv$. We
allow the first $\lfloor\theta m\rfloor$ components of $\vv$ to be adaptable to the
realization of $\xi$, while keeping the remaining $m-\lfloor\theta m\rfloor$
variables independent of the realization of $\xi$, where $\theta \in
[0,1]$. Note that when $\theta=0$, we get the static case SRC. The results are
summarized in Table \ref{tab:adaptability} and we emphasize here two observations: 
\begin{enumerate}
    \item For each problem instance, as the adaptability ratio $\theta$ increases, the robust value
        decreases significantly, which is reasonable both in theory and
        practice.
    \item As the adaptability ratio $\theta$ increases, the influence of pole-sets on
        the robust value increases. For example, when the adaptability ratio
        $\theta=0.25$, the multipolar robust values of all problem instances remain the
        same with different pole-sets except instance (20,20). When the
        adaptability increases, the robust values of more instances improve as
        the better pole-sets are used, which can be clearly seen when
        $\theta=0.75$  and $\theta=1$. 
    \end{enumerate}
\begin{table}[htb!]
    \caption{The benefit of adaptability (hypercube)}
\label{tab:adaptability}
\smallskip
\begin{tabularx}{\textwidth}{@{}l*4{>{\centering\arraybackslash}X}@{}}
\toprule\noalign{\smallskip}
 (m,n), $\Omega_1$ & $\theta=0.25$ & $\theta=0.5$ & $\theta=0.75$ & $\theta=1$ \\
\noalign{\smallskip} \hline\noalign{\smallskip}
(10,9)           & 23.42         & 19.88        & 16.56         & 12.18      \\
\noalign{\smallskip}
(10,15)          & 32.64         & 24.83        & 21.83         & 16.91      \\
\noalign{\smallskip}
(10,20)          & 46.13         & 37.06        & 32.78         & 24.82      \\
\noalign{\smallskip}
(10,30)          & 59.72         & 47.14        & 41.59         & 31.70      \\
\noalign{\smallskip}
(20,9)           & 45.01         & 36.57        & 28.99         & 23.82      \\
\noalign{\smallskip}
(20,15)          & 73.99         & 58.66        & 46.80         & 39.36      \\
\noalign{\smallskip}
(20,20)          & 85.18         & 73.46        & 59.54         & 47.17      \\
\noalign{\smallskip}
(20,30)          & 137.78        & 118.51       & 99.94         & 77.82      \\
\noalign{\smallskip} \hline\noalign{\smallskip}
(m,n), $\Omega_2$ & $\theta=0.25$ & $\theta=0.5$ & $\theta=0.75$ & $\theta=1$ \\
\noalign{\smallskip} \hline\noalign{\smallskip}
(10,9)           & 23.42         & 19.88        & 15.91         & 10.55      \\
\noalign{\smallskip}
(10,15)          & 32.64         & 24.55        & 21.21         & 15.98      \\
\noalign{\smallskip}
(10,20)          & 46.13         & 36.86        & 32.46         & 24.35      \\
\noalign{\smallskip}
(10,30)          & 59.72         & 46.95        & 41.32         & 31.14      \\
\noalign{\smallskip}
(20,9)           & 45.01         & 36.52        & 27.24         & 22.08      \\
\noalign{\smallskip}
(20,15)          & 73.99         & 58.16        & 45.06         & 36.20      \\
\noalign{\smallskip}
(20,20)          & 84.98         & 72.98        & 58.77         & 45.66      \\
\noalign{\smallskip}
(20,30)          & 137.78        & 117.96       & 99.13         & 76.83      \\
\noalign{\smallskip} \hline\noalign{\smallskip}
(m,n), $\Omega_3$ & $\theta=0.25$ & $\theta=0.5$ & $\theta=0.75$ & $\theta=1$ \\
\noalign{\smallskip} \hline\noalign{\smallskip}
(10,9)           & 23.42         & 19.88        & 15.88         & 10.16      \\
\noalign{\smallskip}
(10,15)          & 32.64         & 24.54        & 20.96         & 15.29      \\
\noalign{\smallskip}
(10,20)          & 46.13         & 36.60        & 31.81         & 23.33      \\
\noalign{\smallskip}
(10,30)          & 59.72         & 46.70        & 40.74         & 30.22      \\
\noalign{\smallskip}
(20,9)           & 45.01         & 36.52        & 27.06         & 22.06      \\
\noalign{\smallskip}
(20,15)          & 73.99         & 58.14        & 44.89         & 34.87      \\
\noalign{\smallskip}
(20,20)          & 84.78         & 72.42        & 57.23         & 40.80      \\
\noalign{\smallskip}
(20,30)          & 137.78        & 116.78       & 96.99         & 73.77      \\
\noalign{\smallskip}
 \bottomrule
\end{tabularx}
\end{table}
\section{Conclusion} \label{sec:conclusion}
In this paper, we have presented a novel approach to handle uncertainty in
optimization problems called the multipolar robust approach, which is based on
a set of poles that are used to approximate the fully adjustable policy by a
set of associated recourse decisions at poles. The approach generalizes the
static approach, the affinely adjustable approach, and the fully adjustable
approach, still we can control its complexity by using the concept of the
shadow matrix and considering a reasonable number of poles. Several algorithms
are proposed for the construction of proper pole-sets for hypercubes and balls.
Comprehensive numerical experiments are carried out to evaluate the performance
of the proposed approach in terms of the robust values, the complexity, and the benefit of
adaptability. In general, the results appear encouraging. 

It would be interesting to investigate further the performance of the
multipolar robust approach on other problems. A systematic study of good
approximations of convex bodies by enclosing polytopes with a limited number
of extreme points should help to alleviate overconservatism and get closer
to the optimal fully adaptable robust value. One can also put more focus
on the approximation of convex bodies from inside using, for example, maximum
volume inscribed polytopes to get better lower bounds of the fully adjustable 
robust value.

While the approach was proposed in the context of a two-stage optimization
problem, it can be adapted to multistage optimization. Multipolar decision
rules can also be considered in stochastic programming. The multipolar approach
might also be combined with finite adaptability or multi-static robustness by
partitioning the uncertainty set into several subsets and considering some
multipolar decision rules for each subset.

\appendix\normalsize
\section*{Appendix: the derivation of \eqref{eq:counterpartball}}
We derive the compact formulation \eqref{eq:counterpartball} of
Section \ref{sec:tract} w.r.t. an ellipsoidal uncertainty set defined by
$\Xi:=\left\{\xi: \norm{\vF\xi}{2} \le 1\right\}$.

For each $i^{th}$ constraint, MRC requires the optimum of the following problem
non-positive. 
\begin{align*}
 \max\limits_{\xi, \vs,\vlambda\ge 0}& \hspace{1em} \vU_i\vu -b_i + \sum\limits_{\omega \in
\Omega}\lambda_{\omega}^\xi \vV_i\vv_{\omega}   \\
    \textrm{s.t.} &\hspace{1em} \norm{\vs}{2}  \le 1, & ~ \\
    &\hspace{1em} \vF\xi  = \vs, & \hspace{1em} \veta_i\in \rit^{n_q}  \\
    &\hspace{1em} \sum\limits_{\omega\in\Omega}\lambda^\xi_\omega\omega =
    \vP\xi, & \hspace{1em} \vsigma_i\in \rit^{n_0} \\ 
          &\hspace{1em} \sum\limits_{\omega\in \Omega}\lambda^\xi_\omega = 1, &
    \hspace{1em} \tau_i\in \rit 
\end{align*}
where $\xi\equiv\lbrack \vU,\vb\rbrack$ and $\veta_i, \vsigma_i, \tau_i$ are dual
multipliers corresponding to each group of constraints. Consider the
corresponding Lagrangian
\begin{align*}
    \mathcal{L}\left(\lambda,\xi,\vs,\veta_i,\tau_i,\vsigma_i\right)
    = ~& \vU_i\vu -b_i + \sum\limits_{\omega \in
    \Omega}\lambda_{\omega}^\xi \vV_i\vv_{\omega} + \veta_i^T\left(\vs- \vF
    \xi\right)  \\
    &+
    \vsigma_i^T\left(\sum\limits_{\omega\in \Omega}\lambda^\xi_\omega\omega -\vP\xi\right)+\tau_i\left(\sum\limits_{\omega\in \Omega}\lambda^\xi_\omega -1\right).
\end{align*}
The dual function is then
$\max\limits_{\lambda,\xi,\norm{\vs}{2}\le 1} ~\mathcal{L}
\left(\vlambda,\xi,\veta_i,\tau_i,\vsigma_i\right)$.
Setting the derivative w.r.t. $\vlambda,\xi$ leads to the dual constrains 
\begin{equation}\label{eq:constraints}
\begin{split}
    \vV_i\vv_\omega+ \tau_i - \omega^T\vsigma_i &\le 0 ,\\
    \vF^T\eta_i  - \vL_i &= \allzeros,
\end{split}
\end{equation}
where $\vP=\left[\vP_1,\dots, \vP_k,\dots,\vP_m\right], \vL_i = \begin{pmatrix} \vL_{i1}, \dots ,
\vL_{im}\end{pmatrix}, \vL_{ij} = \delta_{ij}\begin{pmatrix} \vu , -1\end{pmatrix} +
\vP_j^T\vsigma_i, ~j=1,\dots,m$. The dual objective is
\begin{equation}\label{eq:obj}
   \min\limits_{\veta_i,\tau_i,\vsigma_i} \hspace{1em}\norm{\veta_i}{2} -\tau_i.
\end{equation}
By duality, the optimum of the above dual problem is equal to the optimum of
the problem. Thus restricting the non-positivity of the
primal optimum can be equivalently represented as 
\begin{equation}
\begin{split}
   \norm{\veta_i}{2}+\vV_i\vv_\omega - \omega^T\vsigma_i &\le 0 ,\\
    \vF^T\eta_i  - \vL_i &=\allzeros,\\
\veta_i \in \rit^{n_q}, \vsigma_i& \in \mathbb{R}^{n_0}. 
\end{split}
\end{equation}



\end{document}